\documentclass[1p,12pt]{elsarticle}

\usepackage{graphicx}
\usepackage{amsmath,amsxtra,amssymb,latexsym, amscd,amsthm}

\theoremstyle{plain}
\newtheorem{thm}{Theorem}[section]
\newtheorem{cor}[thm]{Corollary}

\newtheorem{lem}[thm]{Lemma}

\theoremstyle{definition}

\newtheorem{rmk}[thm]{Remark}

\numberwithin{equation}{section}
\numberwithin{figure}{section}
\numberwithin{table}{section}

\newcommand{\T}{\operatorname{\mathbb{T}}}
\newcommand{\M}{\operatorname{M}}

\newcommand{\up}{\operatorname{up}}
\newcommand{\down}{\operatorname{down}}
\newcommand{\level}{\operatorname{level}}

\begin{document}

\begin{frontmatter}
\title{\textbf{Double Aztec Rectangles}}


\author{Tri Lai\corref{cor1}\fnref{myfootnote1}}
\address{Institute for Mathematics and its Applications\\ University of Minnesota\\ Minneapolis, MN 55455}
\fntext[myfootnote1]{This research was supported in part by the Institute for Mathematics and its Applications with funds provided by the National Science Foundation (grant no. DMS-0931945).}
\cortext[cor1]{Corresponding author, email: tmlai@ima.umn.edu, tel: 612-626-8319}

\begin{abstract}
We investigate the connection between lozenge tilings and domino tilings by introducing a new family of regions obtained by attaching two different Aztec rectangles.
We prove a simple product formula for the generating functions of the tilings of the new regions, which involves the statistics as in the Aztec diamond theorem (Elkies, Kuperberg, Larsen, and Propp, J. Algebraic Combin. 1992). Moreover, we consider the connection between
the generating function and MacMahon's $q$-enumeration of plane partitions fitting in a given box
\end{abstract}


\end{frontmatter}

\begin{abstract}
We investigate the connection between lozenge tilings and domino tilings by introducing a new family of regions obtained by attaching two different Aztec rectangles.
We prove a simple product formula for the generating functions of the tilings of the new regions, which involves the statistics as in the Aztec diamond theorem (Elkies, Kuperberg, Larsen, and Propp, J. Algebraic Combin. 1992). Moreover, we consider the connection between
the generating function and MacMahon's $q$-enumeration of plane partitions fitting in a given box.
  \bigskip

\end{abstract}

\section{Introduction}

A \emph{plane partition} is a rectangular array of non-negative integers so that all columns are weakly decreasing from top to bottom and  all rows are weakly decreasing from left to right.
A plane partition having $a$ rows and $b$ columns with entries  at most $c$ is identified with it $3$-D interpretation --- a stack of unit cubes fitting in an $a \times b \times c$ box.
The latter stack in turn corresponds to a lozenge tiling of a centrally symmetric hexagon of side-lengths $a,b,c,a,b,c$ (in clockwise order, starting from the northwest side) on the triangular
 lattice. We denote this hexagon by $H_{a,b,c}$ (see Figure \ref{exampleAR}(a) for an example). Here, a \emph{lozenge} (or \emph{unit rhombus}) is union of any two unit equilateral triangles
 sharing an edge; and a \emph{lozenge tiling} of a region (on the triangular lattice) is a covering of the region by lozenges so that there are no gaps or overlaps.

\begin{figure}\centering
\includegraphics[width=14cm]{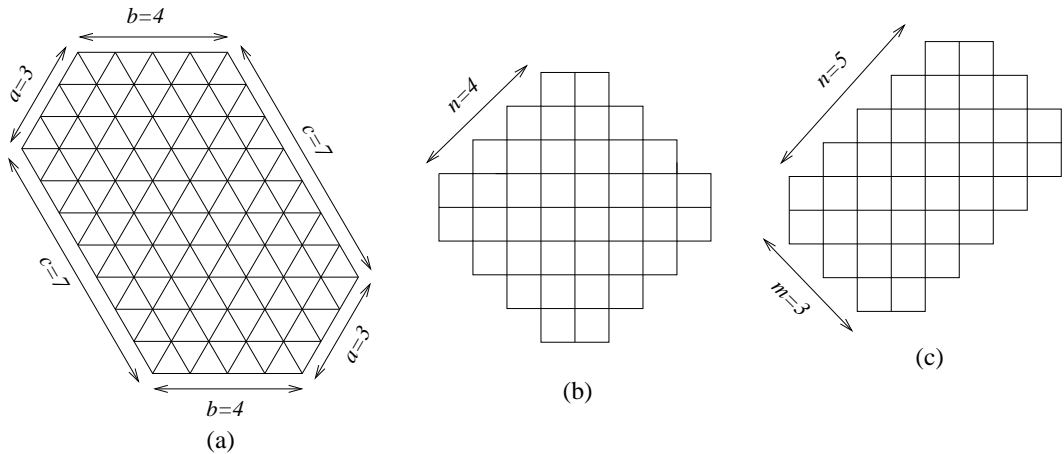}
\caption{(a) The hexagon $H_{3,4,7}$. (b) The Aztec diamond $\mathcal{AD}_{4}$. (c) The Aztec rectangle $\mathcal{AR}_{4,6}$.}
\label{exampleAR}
\end{figure}

Let $q$ be an indeterminate. The \emph{$q$-integer} $[n]_q$ is defined as $[n]_q:=1+q+\dotsc+q^{n-1}$. MacMahon \cite{Mac} proved that
\begin{equation}\label{qMac}
\sum_{\pi}q^{|\pi|}=\prod_{i=1}^{a}\prod_{j=1}^{b}\prod_{k=1}^{c}\frac{[i+j+k-1]_q}{[i+j+t-2]_q},
\end{equation}
where the sum is taken over all plane partitions $\pi$ fitting in an $a\times b \times c$ box and where $|\pi|$ is the number of unit cubes in $\pi$ (i.e. the \emph{volume} of $\pi$).
By letting $q= 1$, this deduces that
\begin{equation}\label{Maceq}
\T\big(H_{a,b,c}\big)=\prod_{i=1}^{a}\prod_{j=1}^{b}\prod_{t=1}^{c}\frac{i+j+t-1}{i+j+t-2},
\end{equation}
where we use the notation $\T(R)$ for the number of tilings of a region $R$.

We now consider regions on a different lattice, the square lattice.  On this lattice, we are interested in \emph{domino tilings}--coverings of a region by dominoes so that there are no gaps or overlaps.
 Here, a \emph{domino} is union of any two unit square sharing an edge. Two central objects in enumeration of domino tilings are the Aztec diamond (see Figure \ref{exampleAR}(b) for the Aztec diamond of order $4$)
 and its natural generalization, the Aztec rectangle (see Figure \ref{exampleAR}(c) for an Aztec rectangle of order $3\times 5$). We denote by $\mathcal{AD}_n$ the Aztec diamond of order $n$,
 and $\mathcal{AR}_{m,n}$ the Aztec rectangle of order $m\times n$.
  One of the crucial results in enumeration of domino tilings is Aztec diamond theorem by Elkies, Kuperberg, Larsen and Propp \cite{Elkies1,Elkies2}, which will stated in the next paragraph.

  Let $T$ be a domino tiling of $\mathcal{AD}_n$. We denote by $r(T)$ is the smallest number of elementary moves ($90^{\circ}$ rotation of a $2\times 2$ block consisting of two vertical or horizontal dominoes as in
  Figure \ref{Drawdominob}) to obtain the tiling $T$ from the \emph{minimal} tiling $T_0$---the tiling consisting of all horizontal dominoes.
  We call $r(T)$ the \emph{rank} of the tiling $T$. For example, the tilings in Figures \ref{ranktiling}(b), (c), (d) have ranks 1,2,5, respectively. We  are also interested in half number of vertical dominoes in $T$, denoted by $v(T)$. The Aztec diamond theorem
  says that the generating function of domino tilings of an Aztec diamond with the two statistics $r(T)$ and $v(T)$ is given by a simple product formula.
\begin{figure}\centering
\includegraphics[width=3cm]{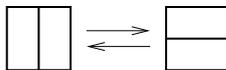}
\caption{The elementary moves: rotation of a $2\times 2$ block consisting of two vertical or horizontal dominoes.}
\label{Drawdominob}
\end{figure}

\begin{figure}\centering
\includegraphics[width=12cm]{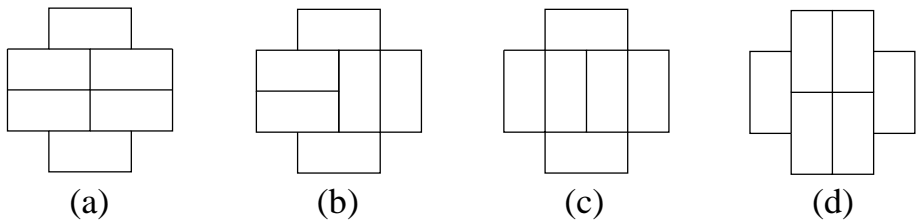}
\caption{Obtaining the other tilings of $\mathcal{AD}_2$ from the minimal tiling on the left by using the elementary moves.}
\label{ranktiling}
\end{figure}

\begin{thm}[Aztec Diamond Theorem \cite{Elkies1,Elkies2}]\label{Aztecthm} For any positive integer $n$ and indeterminates $t$ and $q$
\begin{equation}\label{Azteceq}
\sum_{T}t^{v(T)}q^{r(T)}=\prod_{k=0}^{n-1}(1+tq^{2k+1})^{n-k},
\end{equation}
where the sum is taken over all tilings $T$ of the Aztec diamond of order $n$.
\end{thm}
The $q=t=1$ specialization of Theorem \ref{Aztecthm} implies that the number of domino tilings of the Aztec diamond of  order $n$ is equal to $2^{n(n+1)/2}$.

It seems that the enumeration of domino tilings and the enumeration of lozenge tilings are two isolated topics.
There is not many connections between these two subfields. In the effort to find the connection between these two tiling enumerations, we introduce a new family
 of regions on the square lattice, which are obtained by attaching two different Aztec rectangles (see Figure \ref{doublerectangle} for an example).
 We call these new regions \emph{double Aztec rectangles} (we will present the precise definition of a double Aztec rectangle in the next section).

\begin{figure}\centering
\begin{picture}(0,0)%
\includegraphics{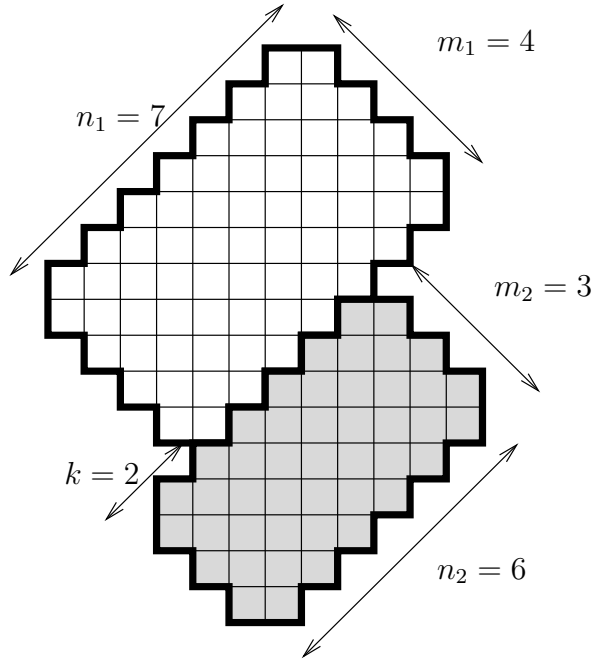}%
\end{picture}
\setlength{\unitlength}{3947sp}%
\begingroup\makeatletter\ifx\SetFigFont\undefined%
\gdef\SetFigFont#1#2#3#4#5{%
  \reset@font\fontsize{#1}{#2pt}%
  \fontfamily{#3}\fontseries{#4}\fontshape{#5}%
  \selectfont}%
\fi\endgroup%
\begin{picture}(3506,4128)(230,-3570)
\put(592,-224){\makebox(0,0)[lb]{\smash{{\SetFigFont{12}{14.4}{\rmdefault}{\mddefault}{\updefault}{$n_1=7$}%
}}}}
\put(2836,249){\makebox(0,0)[lb]{\smash{{\SetFigFont{12}{14.4}{\rmdefault}{\mddefault}{\updefault}{$m_1=4$}%
}}}}
\put(3190,-1287){\makebox(0,0)[lb]{\smash{{\SetFigFont{12}{14.4}{\rmdefault}{\mddefault}{\updefault}{$m_2=3$}%
}}}}
\put(2836,-3059){\makebox(0,0)[lb]{\smash{{\SetFigFont{12}{14.4}{\rmdefault}{\mddefault}{\updefault}{$n_2=6$}%
}}}}
\put(520,-2468){\makebox(0,0)[lb]{\smash{{\SetFigFont{12}{14.4}{\rmdefault}{\mddefault}{\updefault}{$k=2$}%
}}}}
\end{picture}
\caption{The double rectangle $\mathcal{DR}_{3,6,2}^{4,7}$ is obtained by matching two Aztec rectangles $\mathcal{AR}_{4,7}$ (white) and $\mathcal{AR}_{3,6}$ (shaded).}
\label{doublerectangle}
\end{figure}

 We also investigate the generating function of the domino tilings of the double Aztec rectangles involving the two statistics $r(T)$ and $v(T)$ as in the case of the Aztec diamonds.
 We prove that the generating function is given by a simple product formula (see Theorem \ref{mainweight}).
 More surprisingly, the latter formula is roughly the product of an instance of the formula on the right-hand side of (\ref{Azteceq}) in Aztec diamond Theorem and an instance of the expression on the right-hand side of  (\ref{qMac}) in MacMahon's Theorem.  This means that one can view our family of double Aztec rectangles as a ``bridge" connecting the enumerations of domino tilings and lozenge tilings.

 \section{Rank of a tiling, the minimal tiling, and the statement of main result}

First, we give a precise definition of a double Aztec diamond as follows.

\medskip

 Consider two Aztec rectangles $\mathcal{AR}_{m_1,n_1}$ and $\mathcal{AR}_{m_2,n_2}$ with $m_1\leq n_2$ and $m_2\leq n_2$.
  We match the southeast side of $\mathcal{AR}_{m_1,n_1}$ to the northwest side  of $\mathcal{AR}_{m_2,n_2}$ so that the first square
  on the southeast side of $\mathcal{AR}_{m_1,n_1}$ stays immediately on the right of the $(k+1)$-th square of the northwest side of $\mathcal{AR}_{m_2,n_2}$.
  We denote by $\mathcal{DR}_{m_1,n_1,k}^{m_2,n_2}$ the resulting region (see Figure  \ref{doublerectangle}; the white Aztec rectangle indicates $\mathcal{AR}_{m_1,n_1}$
  and the shaded one indicates $\mathcal{AR}_{m_2,n_2}$). In addition, we call $\mathcal{AR}_{m_1,n_1}$ and $\mathcal{AR}_{m_2,n_2}$
    the \textit{upper} and \textit{lower parts} of the double Aztec rectangle, respectively.

\begin{figure}\centering
\includegraphics[width=9cm]{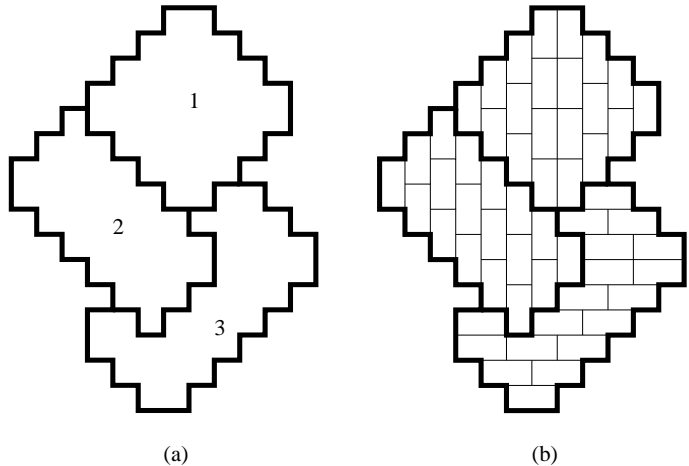}
\caption{The minimal tiling of the  double Aztec rectangle $\mathcal{DR}_{3,6,2}^{4,7}$.}
\label{double6}
\end{figure}

We notice that, in general, the double Aztec rectangle $D:=\mathcal{DR}_{m_1,n_1,k}^{m_2,n_2}$ does \textit{not} admit a tiling consisting of all horizontal dominoes. This means that we need another candidate for the minimal tiling of the double Aztec rectangle as follows.

We divide the region $D$ into 3 disjoint parts (labelled by $1,2,3$) as in Figure \ref{double6}(a). In precise, part $1$ is an Aztec diamond of order $m_1$ on the top of the region;
 part $2$ is an Aztec rectangle $\mathcal{AR}_{m_1+k,n_1-m_1}$, where the top and the bottom are shrunk to size 1; and part $3$ is the remaining.
Next, we cover the parts $1$ and $2$ by vertical dominoes, and the part $3$ by horizontal dominoes (see Figure \ref{double6}(b)). We call the resulting tiling $T_0$ the \emph{minimal tiling}
  of the double Aztec rectangle.   Similar to the case of Aztec diamonds, we now define the \emph{rank} $r(T)$ of a tiling $T$ of the region $D$ to be the smallest number of elementary moves to obtain
  $T$ from $T_0$. We also use the notation $v(T)$ for one half of the number of vertical dominoes in $T$. 
  
\begin{rmk}
The detailed explanation for the choice of the tiling $T_0$ as the minimal tiling will be shown in Section 5.  However, we can explain intuitively as follows. Similar to the correspondence of Fu and Eu in the case of Aztec diamonds \cite{Eu}, each domino tiling $T$ of the double
  Aztec rectangle are in bijection with a family of non-intersecting lattice paths $\textbf{P}_T$. Moreover, we will show that the difference between the rank of the tiling $T$ and the 
  total area underneath the lattice paths in the family $\textbf{P}_T$ is a constant, and that and the tiling $T_0$ corresponds to the family having the smallest underneath area. This justifies the choice of the minimal tiling $T_0$.
\end{rmk}
  
The generating function of domino tilings of the double Aztec rectangle is given by the
  theorem stated below.

\begin{thm}\label{mainweight}
Assume that $m_1,m_2,n_1,n_2$ are positive integers, and $k$ is a non-negative integer so that $m_1\leq n_1$, $m_2\leq n_2$, $k\leq \min (m_2,n_2-1)$, and $n_1-m_1=n_2-m_2$. Then
\begin{align}
\sum_{T}t^{r(T)}q^{v(T)}&=t^{\binom{m_1+1}{2}+\binom{m_2+1}{2}+(n_1-m_1)(m_1+k)/2}q^{N+(n_1-m_1)(m_1+k)+A}\notag\\
&\times\prod_{i=0}^{m_1-1}(\bigstar'_i)^{m_1-i}\prod_{i=0}^{m_2-1}(\bigstar_i)^{m_2-i}\prod_{i=1}^{n_1-m_1}\prod_{j=1}^{m_2-k+1}\prod_{t=1}^{m_1+k}\frac{[i+j+t-1]_{q^2}}{[i+j+t-2]_{q^2}},
\end{align}
where the sum is taken over all tilings $T$ of the double Aztec rectangle $\mathcal{DR}_{m_1,n_1,k}^{m_2,n_2}$,  where
\begin{align*}
N=&m_1(m_1+1)(n_1-1)-m_2(m_2+1)(n_2-1)\notag\\
&+(n_1-m_1)(2m_2^2+m_2m_1+m_2n_1+k^2+2km_1+m_1n_1+k-m_2),
\end{align*}
and
\begin{align*}
A=&\frac{2m_2(m_2-1)(m_2+1)}{3}+(m_2-k+1)(m_2+n_2-1)(n_1-m_1)\notag\\
&+\frac{m_1(m_1+1)(2k+2m_1+2n_2-1)}{2},\end{align*}
and where $\bigstar_i=q^{2m_2+2n_2-3}(1+t^{-1}q^{-2i-1})$ and $\bigstar_i'=q^{2m_2+2k+1}(1+t^{-1}q^{2i+1})$.
\end{thm}

By letting $t=q=1$, we deduce the following elegant corollary  from Theorem \ref{mainweight}.
\begin{cor}\label{coroweight}
Assume that $m_1,m_2,n_1,n_2$ are positive integers, and $k$ is a non-negative integer so that $m_1\leq n_1$, $m_2\leq n_2$, $k\leq \min (m_2,n_2-1)$, and $n_1-m_1=n_2-m_2$. Then
\begin{equation}\label{coroeq}
\T\left(\mathcal{DR}_{m_1,n_1,k}^{m_2,n_2}\right)=2^{\binom{m_1+1}{2}+\binom{m_2+1}{2}}\T\left(H_{n_1-m_1,m_2-k+1,m_1+k}\right).
\end{equation}
\end{cor}
Equation (\ref{coroeq}) shows an interesting connecting between the two types of tilings: domino tilings (on the left-hand side) and lozenge tilings  (on the right-hand side). 
We recommend the reader to \cite{Tri2} and \cite{Tri3} for  more results of the same flavor.

The goal of this paper is to prove Theorem \ref{mainweight}. The rest of this paper is organized as follows. In Section 3, we introduce several fundamental results in the subgraph replacement method.
These fundamental results will be employed to prove the key lemma of  the paper (Lemma \ref{weighttransform}). Next, in Section 4, we give an exact formula for a weighted sum of the domino tilings of 
 a double Aztec rectangle (see Theorem \ref{weightdouble}). Finally, in Section 5,  we present a proof of  Theorem \ref{mainweight} by using Theorem \ref{weightdouble}.
 
\section{Subgraph replacements}

In this section, we introduce a powerful method in enumerating of tilings, the subgraph replacement method.

A general lattice divides the plane into disjoint \emph{fundamental regions} (which are unit squares on the square lattice, and are unit equilateral triangles on the triangular lattice). We define a \emph{tile} to be the union of any
two fundamental regions sharing an edge (which is a domino on the square lattice, and a lozenge on the triangular lattice); and the \emph{tiling} of a region in a covering of the region by tiles so that there are no gaps and overlaps as usual.  In this section we allow tiles of a region carry weights. We use the notation $\M(R)$ for the sum of weights of all tilings in $R$,
where the \emph{weight} a tiling is the product of weights of its tiles. We call $\M(R)$ the \emph{tiling generating function} of $R$. In the unweighted case, $\M(R)$ is exactly the number of tilings of $R$.

 A \emph{perfect matching} of a graph $G$ is a collection of disjoint edges covering all vertices of $G$. The tilings
of a region $R$ can be identified with the perfect matchings of its \emph{dual graph} (the graph whose vertices are fundamental regions in $R$ and whose edges connect precisely two fundamental regions sharing an edge).
If $R$ is a weighted
region, we assume that each edge of its dual graph caries the same weight as the corresponding tile. In the view of this, we use notation $\M(G)$
 for the sum of weight of all perfect matchings of the weighted graph $G$, where the weight a perfect matching is the product of weights of its edges. $\M(G)$ is called the \emph{matching generating function} of $G$.

Next, we present several fundamental replacement rules, which allow us replace a subgraph of a graph $G$ by a new subgraph, so that $\M(G)$ does not change or changes in a predictable way.

A \emph{forced edge} of a (weighted) graph $G$ is an edge contained in any perfect matchings of $G$. If we remove several forced edges $e_1,e_2,\dotsc,e_k$ from $G$, then we get a new graph $G'$ and
\begin{equation*}
\M(G)=\M(G')\prod_{i=1}^kwt(e_i),
\end{equation*}
where $wt(e_i)$ is the weight of the edge $e_i$.

\begin{figure}\centering
\includegraphics{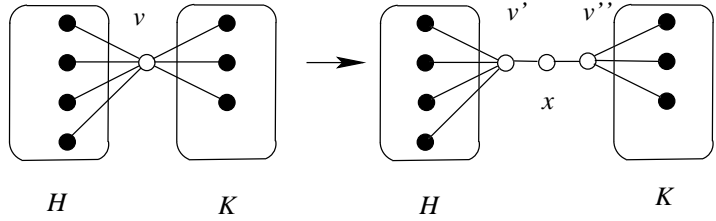}
\caption{Vertex splitting rule}\label{VSfig}
\end{figure}
The following useful lemma is a special case of Lemma 1.3 in \cite{Ciucu1}.
\begin{lem} [Vertex-Splitting Lemma]\label{VS}
 Let $G$ be a graph, $v$ be a vertex of it, and denote the set of neighbors of $v$ by $N(v)$.
  For any disjoint union $N(v)=H\cup K$, let $G'$ be the graph obtained from $G\setminus v$
  by including three new vertices $v'$, $v''$ and $x$ so that $N(v')=H\cup \{x\}$, $N(v'')=K\cup\{x\}$, and $N(x)=\{v',v''\}$ (see Figure \ref{VSfig}). Then $\M(G)=\M(G')$.
\end{lem}

\begin{lem}[Star Lemma; Lemma 3.2 in \cite{Tri1}]\label{star}
Let $G$ be a weighted graph, and let $v$ be a vertex of~$G$. Let $G'$ be the graph obtained from $G$ by multiplying the weights of all edges incident to $v$ by $t>0$. Then $\M(G')=t\M(G)$.
\end{lem}

The following result is a generalization (due to Propp) of the ``urban renewal" trick first observed by Kuperberg (see \cite[Section 5]{propp}).
\begin{figure}\centering
\begin{picture}(0,0)%
\includegraphics{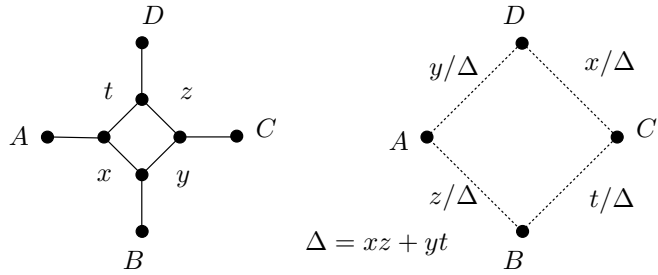}%
\end{picture}%
\setlength{\unitlength}{3947sp}%
\begingroup\makeatletter\ifx\SetFigFont\undefined%
\gdef\SetFigFont#1#2#3#4#5{%
  \reset@font\fontsize{#1}{#2pt}%
  \fontfamily{#3}\fontseries{#4}\fontshape{#5}%
  \selectfont}%
\fi\endgroup%
\begin{picture}(4054,1735)(340,-948)
\put(355,-156){\makebox(0,0)[lb]{\smash{{\SetFigFont{10}{12.0}{\familydefault}{\mddefault}{\updefault}{$A$}%
}}}}
\put(1064,-933){\makebox(0,0)[lb]{\smash{{\SetFigFont{10}{12.0}{\familydefault}{\mddefault}{\updefault}{$B$}%
}}}}
\put(1891,-106){\makebox(0,0)[lb]{\smash{{\SetFigFont{10}{12.0}{\familydefault}{\mddefault}{\updefault}{$C$}%
}}}}
\put(1182,603){\makebox(0,0)[lb]{\smash{{\SetFigFont{10}{12.0}{\familydefault}{\mddefault}{\updefault}{$D$}%
}}}}
\put(2717,-189){\makebox(0,0)[lb]{\smash{{\SetFigFont{10}{12.0}{\familydefault}{\mddefault}{\updefault}{$A$}%
}}}}
\put(3426,-933){\makebox(0,0)[lb]{\smash{{\SetFigFont{10}{12.0}{\familydefault}{\mddefault}{\updefault}{$B$}%
}}}}
\put(4253,-106){\makebox(0,0)[lb]{\smash{{\SetFigFont{10}{12.0}{\familydefault}{\mddefault}{\updefault}{$C$}%
}}}}
\put(3426,603){\makebox(0,0)[lb]{\smash{{\SetFigFont{10}{12.0}{\familydefault}{\mddefault}{\updefault}{$D$}%
}}}}
\put(904,-382){\makebox(0,0)[lb]{\smash{{\SetFigFont{10}{12.0}{\familydefault}{\mddefault}{\updefault}{$x$}%
}}}}
\put(1396,-388){\makebox(0,0)[lb]{\smash{{\SetFigFont{10}{12.0}{\familydefault}{\mddefault}{\updefault}{$y$}%
}}}}
\put(1418,130){\makebox(0,0)[lb]{\smash{{\SetFigFont{10}{12.0}{\familydefault}{\mddefault}{\updefault}{$z$}%
}}}}
\put(946,130){\makebox(0,0)[lb]{\smash{{\SetFigFont{10}{12.0}{\familydefault}{\mddefault}{\updefault}{$t$}%
}}}}
\put(2968,284){\makebox(0,0)[lb]{\smash{{\SetFigFont{10}{12.0}{\familydefault}{\mddefault}{\updefault}{$y/\Delta$}%
}}}}
\put(3934,311){\makebox(0,0)[lb]{\smash{{\SetFigFont{10}{12.0}{\familydefault}{\mddefault}{\updefault}{$x/\Delta$}%
}}}}
\put(3964,-544){\makebox(0,0)[lb]{\smash{{\SetFigFont{10}{12.0}{\familydefault}{\mddefault}{\updefault}{$t/\Delta$}%
}}}}
\put(2965,-526){\makebox(0,0)[lb]{\smash{{\SetFigFont{10}{12.0}{\familydefault}{\mddefault}{\updefault}{$z/\Delta$}%
}}}}
\put(2197,-817){\makebox(0,0)[lb]{\smash{{\SetFigFont{10}{12.0}{\familydefault}{\mddefault}{\updefault}{$\Delta= xz+yt$}%
}}}}
\end{picture}%
\caption{Urban renewal.}
\label{spider1}
\end{figure}

\begin{lem} [Spider Lemma]\label{spider}
 Let $G$ be a weighted graph containing the subgraph $K$ shown on the left in Figure \ref{spider1} (the labels indicate weights, unlabeled edges have weight 1). Suppose that $\Delta=xz+yt\not=0$.
 Suppose in addition that the four inner black vertices in the subgraph $K$, different from $A,B,C,D$, have no neighbors outside $K$. Let $G'$ be the graph obtained from
 $G$ by replacing $K$ by the graph $\overline{K}$ shown on right in Figure \ref{spider1}, where the dashed lines indicate new edges, weighted as shown. Then $\M(G)=(xz+yt)\M(G')$.
\end{lem}


\begin{figure}\centering
\resizebox{!}{5.5cm}{
\begin{picture}(0,0)%
\includegraphics{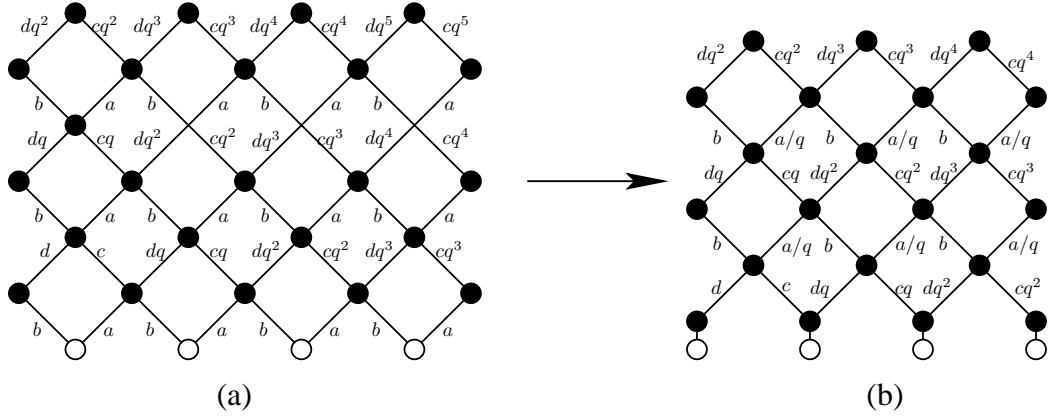}%
\end{picture}%
%
%
\setlength{\unitlength}{3947sp}%
\begingroup\makeatletter\ifx\SetFigFont\undefined%
\gdef\SetFigFont#1#2#3#4#5{%
  \reset@font\fontsize{#1}{#2pt}%
  \fontfamily{#3}\fontseries{#4}\fontshape{#5}%
  \selectfont}%
\fi\endgroup%
\begin{picture}(8716,3496)(1084,-3504)
\put(1300,-2822){\makebox(0,0)[lb]{\smash{{\SetFigFont{10}{12.0}{\rmdefault}{\mddefault}{\updefault}{$b$}%
}}}}
\put(1891,-2822){\makebox(0,0)[lb]{\smash{{\SetFigFont{10}{12.0}{\rmdefault}{\mddefault}{\updefault}{$a$}%
}}}}
\put(2245,-2822){\makebox(0,0)[lb]{\smash{{\SetFigFont{10}{12.0}{\rmdefault}{\mddefault}{\updefault}{$b$}%
}}}}
\put(2836,-2822){\makebox(0,0)[lb]{\smash{{\SetFigFont{10}{12.0}{\rmdefault}{\mddefault}{\updefault}{$a$}%
}}}}
\put(3190,-2822){\makebox(0,0)[lb]{\smash{{\SetFigFont{10}{12.0}{\rmdefault}{\mddefault}{\updefault}{$b$}%
}}}}
\put(3781,-2822){\makebox(0,0)[lb]{\smash{{\SetFigFont{10}{12.0}{\rmdefault}{\mddefault}{\updefault}{$a$}%
}}}}
\put(4135,-2822){\makebox(0,0)[lb]{\smash{{\SetFigFont{10}{12.0}{\rmdefault}{\mddefault}{\updefault}{$b$}%
}}}}
\put(4726,-2822){\makebox(0,0)[lb]{\smash{{\SetFigFont{10}{12.0}{\rmdefault}{\mddefault}{\updefault}{$a$}%
}}}}
\put(1359,-2152){\makebox(0,0)[lb]{\smash{{\SetFigFont{10}{12.0}{\rmdefault}{\mddefault}{\updefault}{$d$}%
}}}}
\put(1832,-2152){\makebox(0,0)[lb]{\smash{{\SetFigFont{10}{12.0}{\rmdefault}{\mddefault}{\updefault}{$c$}%
}}}}
\put(2245,-2152){\makebox(0,0)[lb]{\smash{{\SetFigFont{10}{12.0}{\rmdefault}{\mddefault}{\updefault}{$dq$}%
}}}}
\put(2777,-2152){\makebox(0,0)[lb]{\smash{{\SetFigFont{10}{12.0}{\rmdefault}{\mddefault}{\updefault}{$cq$}%
}}}}
\put(3131,-2152){\makebox(0,0)[lb]{\smash{{\SetFigFont{10}{12.0}{\rmdefault}{\mddefault}{\updefault}{$dq^2$}%
}}}}
\put(3722,-2152){\makebox(0,0)[lb]{\smash{{\SetFigFont{10}{12.0}{\rmdefault}{\mddefault}{\updefault}{$cq^2$}%
}}}}
\put(4076,-2152){\makebox(0,0)[lb]{\smash{{\SetFigFont{10}{12.0}{\rmdefault}{\mddefault}{\updefault}{$dq^3$}%
}}}}
\put(4666,-2152){\makebox(0,0)[lb]{\smash{{\SetFigFont{10}{12.0}{\rmdefault}{\mddefault}{\updefault}{$cq^3$}%
}}}}
\put(1314,-1856){\makebox(0,0)[lb]{\smash{{\SetFigFont{10}{12.0}{\rmdefault}{\mddefault}{\updefault}{$b$}%
}}}}
\put(1905,-1856){\makebox(0,0)[lb]{\smash{{\SetFigFont{10}{12.0}{\rmdefault}{\mddefault}{\updefault}{$a$}%
}}}}
\put(2259,-1856){\makebox(0,0)[lb]{\smash{{\SetFigFont{10}{12.0}{\rmdefault}{\mddefault}{\updefault}{$b$}%
}}}}
\put(2850,-1856){\makebox(0,0)[lb]{\smash{{\SetFigFont{10}{12.0}{\rmdefault}{\mddefault}{\updefault}{$a$}%
}}}}
\put(3204,-1856){\makebox(0,0)[lb]{\smash{{\SetFigFont{10}{12.0}{\rmdefault}{\mddefault}{\updefault}{$b$}%
}}}}
\put(3795,-1856){\makebox(0,0)[lb]{\smash{{\SetFigFont{10}{12.0}{\rmdefault}{\mddefault}{\updefault}{$a$}%
}}}}
\put(4149,-1856){\makebox(0,0)[lb]{\smash{{\SetFigFont{10}{12.0}{\rmdefault}{\mddefault}{\updefault}{$b$}%
}}}}
\put(4740,-1856){\makebox(0,0)[lb]{\smash{{\SetFigFont{10}{12.0}{\rmdefault}{\mddefault}{\updefault}{$a$}%
}}}}
\put(1314,-912){\makebox(0,0)[lb]{\smash{{\SetFigFont{10}{12.0}{\rmdefault}{\mddefault}{\updefault}{$b$}%
}}}}
\put(1905,-912){\makebox(0,0)[lb]{\smash{{\SetFigFont{10}{12.0}{\rmdefault}{\mddefault}{\updefault}{$a$}%
}}}}
\put(2259,-912){\makebox(0,0)[lb]{\smash{{\SetFigFont{10}{12.0}{\rmdefault}{\mddefault}{\updefault}{$b$}%
}}}}
\put(2850,-912){\makebox(0,0)[lb]{\smash{{\SetFigFont{10}{12.0}{\rmdefault}{\mddefault}{\updefault}{$a$}%
}}}}
\put(3204,-912){\makebox(0,0)[lb]{\smash{{\SetFigFont{10}{12.0}{\rmdefault}{\mddefault}{\updefault}{$b$}%
}}}}
\put(3795,-912){\makebox(0,0)[lb]{\smash{{\SetFigFont{10}{12.0}{\rmdefault}{\mddefault}{\updefault}{$a$}%
}}}}
\put(4149,-912){\makebox(0,0)[lb]{\smash{{\SetFigFont{10}{12.0}{\rmdefault}{\mddefault}{\updefault}{$b$}%
}}}}
\put(4740,-912){\makebox(0,0)[lb]{\smash{{\SetFigFont{10}{12.0}{\rmdefault}{\mddefault}{\updefault}{$a$}%
}}}}
\put(1254,-1207){\makebox(0,0)[lb]{\smash{{\SetFigFont{10}{12.0}{\rmdefault}{\mddefault}{\updefault}{$dq$}%
}}}}
\put(1832,-1201){\makebox(0,0)[lb]{\smash{{\SetFigFont{10}{12.0}{\rmdefault}{\mddefault}{\updefault}{$cq$}%
}}}}
\put(2140,-1207){\makebox(0,0)[lb]{\smash{{\SetFigFont{10}{12.0}{\rmdefault}{\mddefault}{\updefault}{$dq^2$}%
}}}}
\put(2777,-1201){\makebox(0,0)[lb]{\smash{{\SetFigFont{10}{12.0}{\rmdefault}{\mddefault}{\updefault}{$cq^2$}%
}}}}
\put(3131,-1228){\makebox(0,0)[lb]{\smash{{\SetFigFont{10}{12.0}{\rmdefault}{\mddefault}{\updefault}{$dq^3$}%
}}}}
\put(3675,-1207){\makebox(0,0)[lb]{\smash{{\SetFigFont{10}{12.0}{\rmdefault}{\mddefault}{\updefault}{$cq^3$}%
}}}}
\put(1196,-262){\makebox(0,0)[lb]{\smash{{\SetFigFont{10}{12.0}{\rmdefault}{\mddefault}{\updefault}{$dq^2$}%
}}}}
\put(1787,-262){\makebox(0,0)[lb]{\smash{{\SetFigFont{10}{12.0}{\rmdefault}{\mddefault}{\updefault}{$cq^2$}%
}}}}
\put(2141,-262){\makebox(0,0)[lb]{\smash{{\SetFigFont{10}{12.0}{\rmdefault}{\mddefault}{\updefault}{$dq^3$}%
}}}}
\put(2777,-256){\makebox(0,0)[lb]{\smash{{\SetFigFont{10}{12.0}{\rmdefault}{\mddefault}{\updefault}{$cq^3$}%
}}}}
\put(4076,-1201){\makebox(0,0)[lb]{\smash{{\SetFigFont{10}{12.0}{\rmdefault}{\mddefault}{\updefault}{$dq^4$}%
}}}}
\put(4725,-1201){\makebox(0,0)[lb]{\smash{{\SetFigFont{10}{12.0}{\rmdefault}{\mddefault}{\updefault}{$cq^4$}%
}}}}
\put(3709,-261){\makebox(0,0)[lb]{\smash{{\SetFigFont{10}{12.0}{\rmdefault}{\mddefault}{\updefault}{$cq^4$}%
}}}}
\put(4076,-256){\makebox(0,0)[lb]{\smash{{\SetFigFont{10}{12.0}{\rmdefault}{\mddefault}{\updefault}{$dq^5$}%
}}}}
\put(4725,-256){\makebox(0,0)[lb]{\smash{{\SetFigFont{10}{12.0}{\rmdefault}{\mddefault}{\updefault}{$cq^5$}%
}}}}
\put(3114,-261){\makebox(0,0)[lb]{\smash{{\SetFigFont{10}{12.0}{\rmdefault}{\mddefault}{\updefault}{$dq^4$}%
}}}}
\put(6970,-2468){\makebox(0,0)[lb]{\smash{{\SetFigFont{10}{12.0}{\rmdefault}{\mddefault}{\updefault}{$d$}%
}}}}
\put(7560,-2441){\makebox(0,0)[lb]{\smash{{\SetFigFont{10}{12.0}{\rmdefault}{\mddefault}{\updefault}{$c$}%
}}}}
\put(7796,-2468){\makebox(0,0)[lb]{\smash{{\SetFigFont{10}{12.0}{\rmdefault}{\mddefault}{\updefault}{$dq$}%
}}}}
\put(8505,-2468){\makebox(0,0)[lb]{\smash{{\SetFigFont{10}{12.0}{\rmdefault}{\mddefault}{\updefault}{$cq$}%
}}}}
\put(8741,-2468){\makebox(0,0)[lb]{\smash{{\SetFigFont{10}{12.0}{\rmdefault}{\mddefault}{\updefault}{$dq^2$}%
}}}}
\put(9509,-2468){\makebox(0,0)[lb]{\smash{{\SetFigFont{10}{12.0}{\rmdefault}{\mddefault}{\updefault}{$cq^2$}%
}}}}
\put(6984,-2092){\makebox(0,0)[lb]{\smash{{\SetFigFont{10}{12.0}{\rmdefault}{\mddefault}{\updefault}{$b$}%
}}}}
\put(7914,-2114){\makebox(0,0)[lb]{\smash{{\SetFigFont{10}{12.0}{\rmdefault}{\mddefault}{\updefault}{$b$}%
}}}}
\put(8874,-2092){\makebox(0,0)[lb]{\smash{{\SetFigFont{10}{12.0}{\rmdefault}{\mddefault}{\updefault}{$b$}%
}}}}
\put(6984,-1206){\makebox(0,0)[lb]{\smash{{\SetFigFont{10}{12.0}{\rmdefault}{\mddefault}{\updefault}{$b$}%
}}}}
\put(7929,-1206){\makebox(0,0)[lb]{\smash{{\SetFigFont{10}{12.0}{\rmdefault}{\mddefault}{\updefault}{$b$}%
}}}}
\put(8874,-1206){\makebox(0,0)[lb]{\smash{{\SetFigFont{10}{12.0}{\rmdefault}{\mddefault}{\updefault}{$b$}%
}}}}
\put(6910,-1496){\makebox(0,0)[lb]{\smash{{\SetFigFont{10}{12.0}{\rmdefault}{\mddefault}{\updefault}{$dq$}%
}}}}
\put(7560,-1496){\makebox(0,0)[lb]{\smash{{\SetFigFont{10}{12.0}{\rmdefault}{\mddefault}{\updefault}{$cq$}%
}}}}
\put(7796,-1496){\makebox(0,0)[lb]{\smash{{\SetFigFont{10}{12.0}{\rmdefault}{\mddefault}{\updefault}{$dq^2$}%
}}}}
\put(8505,-1496){\makebox(0,0)[lb]{\smash{{\SetFigFont{10}{12.0}{\rmdefault}{\mddefault}{\updefault}{$cq^2$}%
}}}}
\put(8800,-1502){\makebox(0,0)[lb]{\smash{{\SetFigFont{10}{12.0}{\rmdefault}{\mddefault}{\updefault}{$dq^3$}%
}}}}
\put(9450,-1496){\makebox(0,0)[lb]{\smash{{\SetFigFont{10}{12.0}{\rmdefault}{\mddefault}{\updefault}{$cq^3$}%
}}}}
\put(6851,-492){\makebox(0,0)[lb]{\smash{{\SetFigFont{10}{12.0}{\rmdefault}{\mddefault}{\updefault}{$dq^2$}%
}}}}
\put(7501,-492){\makebox(0,0)[lb]{\smash{{\SetFigFont{10}{12.0}{\rmdefault}{\mddefault}{\updefault}{$cq^2$}%
}}}}
\put(7855,-492){\makebox(0,0)[lb]{\smash{{\SetFigFont{10}{12.0}{\rmdefault}{\mddefault}{\updefault}{$dq^3$}%
}}}}
\put(8446,-492){\makebox(0,0)[lb]{\smash{{\SetFigFont{10}{12.0}{\rmdefault}{\mddefault}{\updefault}{$cq^3$}%
}}}}
\put(8800,-492){\makebox(0,0)[lb]{\smash{{\SetFigFont{10}{12.0}{\rmdefault}{\mddefault}{\updefault}{$dq^4$}%
}}}}
\put(9450,-557){\makebox(0,0)[lb]{\smash{{\SetFigFont{10}{12.0}{\rmdefault}{\mddefault}{\updefault}{$cq^4$}%
}}}}
\put(7560,-2114){\makebox(0,0)[lb]{\smash{{\SetFigFont{10}{12.0}{\rmdefault}{\mddefault}{\updefault}{$a/q$}%
}}}}
\put(7502,-1208){\makebox(0,0)[lb]{\smash{{\SetFigFont{10}{12.0}{\rmdefault}{\mddefault}{\updefault}{$a/q$}%
}}}}
\put(8447,-1208){\makebox(0,0)[lb]{\smash{{\SetFigFont{10}{12.0}{\rmdefault}{\mddefault}{\updefault}{$a/q$}%
}}}}
\put(9392,-1208){\makebox(0,0)[lb]{\smash{{\SetFigFont{10}{12.0}{\rmdefault}{\mddefault}{\updefault}{$a/q$}%
}}}}
\put(8511,-2092){\makebox(0,0)[lb]{\smash{{\SetFigFont{10}{12.0}{\rmdefault}{\mddefault}{\updefault}{$a/q$}%
}}}}
\put(9456,-2092){\makebox(0,0)[lb]{\smash{{\SetFigFont{10}{12.0}{\rmdefault}{\mddefault}{\updefault}{$a/q$}%
}}}}
\end{picture}}
\caption{The transformation in Lemma  \ref{weighttransform}.}
\label{doubletransform}
\end{figure}

\begin{figure}\centering
\includegraphics[width=10cm]{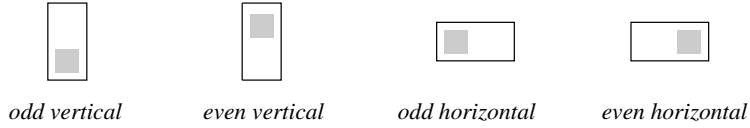}
\caption{Four types of colored dominoes.}
\label{Drawdomino2}
\end{figure}

Next, we color the double Aztec rectangle $\mathcal{DR}^{m_2,n_2}_{m_1,n_1,k}$ by black and white so that two neighbour  unit squares have opposite colors and that the unit squares along the southwest side of $\mathcal{AR}_{m_1,n_1}$ are white (see Figure \ref{double7} for an example). We assign weights to the (colored) dominoes of a double Aztec rectangle as follows. Assume that $a,b,c,d,q$ are five positive numbers.
We assign to each \emph{odd horizontal domino} (see Figure \ref{Drawdomino2} for four types of dominoes) a weight $b$, each \emph{even vertical domino} a weight $a$,
 each \emph{even horizontal domino} on the level $k$ a weight $cq^{k-1}$ (the bottom of the region is on the level $0$), and each \emph{odd vertical domino} on the level $k$ a weight $dq^k$.
 Denote by $wt:=wt_{a,b}^{c,d}(q)$ the resulting weight assignment.

If an Aztec rectangle $\mathcal{AR}_{m,n}$ is assigned the weights as above, then we denote by \\ $AR_{m,n}=AR_{m,n}(a,b,c,d,q)$ its dual graph rotated $45^{\circ}$ clockwise
 (see Figure \ref{doubletransform}(a) for an example).

The \textit{connected sum} $G\#G'$ of two disjoint graphs $G$ and $G'$ along the ordered sets of vertices $\{v_1,\dotsc,v_n\}\subset V(G)$ and $\{v'_1,\dotsc,v'_n\}\subset V(G')$ is the graph obtained from $G$ and $G'$ by identifying vertices $v_i$ and $v'_i$, for $i=1,\dotsc,n$.

Using the above fundamental subgraph replacement rules, we get the following new replacement rule, which will be employed in our proof in the next section.

\begin{lem}\label{weighttransform}
Let $G$ be a weighted graph, and $\{v_1,v_2,\dots,v_n\}$ an ordered set of its vertices. Then
\begin{equation*}
\M\left(G\#AR_{m,n}(a,b,c,d,q)\right)=(ad+bc)^{m}q^{m(n-1)+\binom{m}{2}}\M\left(G\#{}_{|}AR_{m-\frac{1}{2},n-1}(a/q,b,c,d)\right),
\end{equation*}
where ${}_{|}AR_{m-\frac{1}{2},n-1}(a/q,b,c,d,q)$ is obtained from the graph $AR_{m,n-1}(a/q,b,cq,dq,q)$ by removing the bottommost vertices,
and appending a vertical edge to each of the bottommost vertices of the resulting graph; and where the connected sum acts on $G$ along the ordered set $\{v_1,v_2,\dotsc,v_n\}$,
and on the other summands 
along their bottommost vertices ordered from left to right (see Figure \ref{doubletransform}).
\end{lem}

\begin{figure}\centering
\setlength{\unitlength}{3947sp}%
\begingroup\makeatletter\ifx\SetFigFont\undefined%
\gdef\SetFigFont#1#2#3#4#5{%
  \reset@font\fontsize{#1}{#2pt}%
  \fontfamily{#3}\fontseries{#4}\fontshape{#5}%
  \selectfont}%
\fi\endgroup%
\resizebox{14cm}{!}{
\begin{picture}(0,0)%
\includegraphics{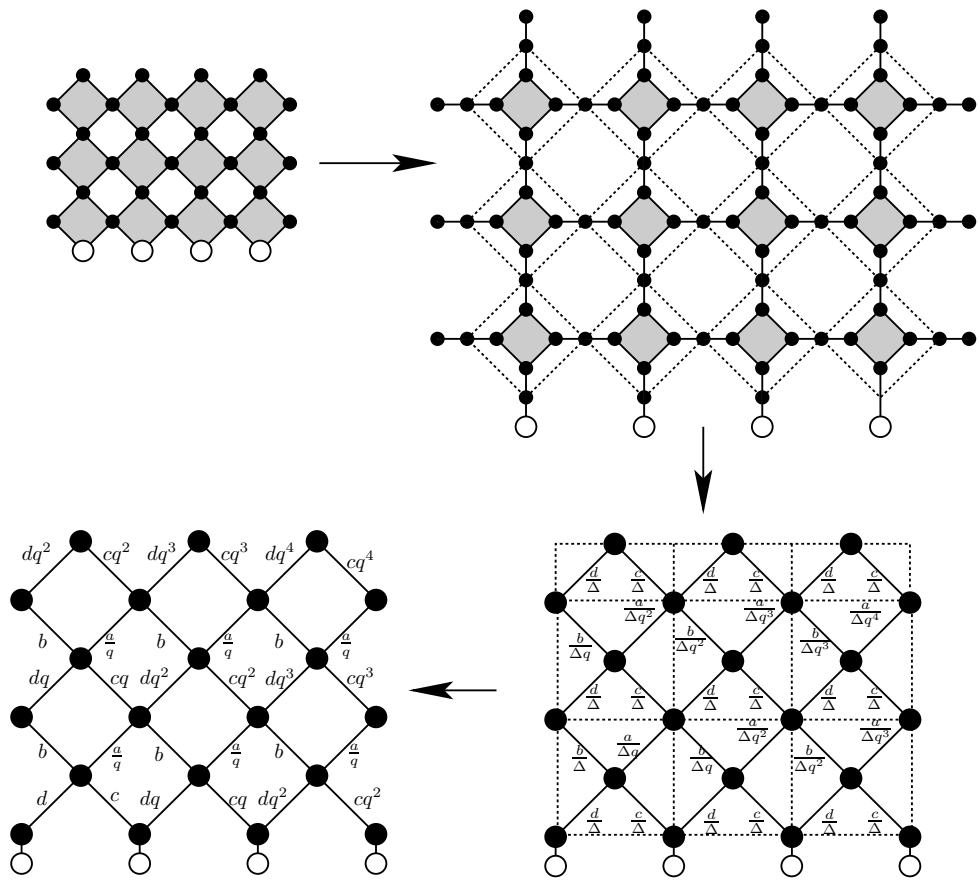}%
\end{picture}%
%
%

\begin{picture}(8450,7014)(592,-6583)
\put(809,-5991){\makebox(0,0)[lb]{\smash{{\SetFigFont{10}{12.0}{\rmdefault}{\mddefault}{\updefault}{$d$}%
}}}}
\put(1399,-5964){\makebox(0,0)[lb]{\smash{{\SetFigFont{10}{12.0}{\rmdefault}{\mddefault}{\updefault}{$c$}%
}}}}
\put(1635,-5991){\makebox(0,0)[lb]{\smash{{\SetFigFont{10}{12.0}{\rmdefault}{\mddefault}{\updefault}{$dq$}%
}}}}
\put(2344,-5991){\makebox(0,0)[lb]{\smash{{\SetFigFont{10}{12.0}{\rmdefault}{\mddefault}{\updefault}{$cq$}%
}}}}
\put(2580,-5991){\makebox(0,0)[lb]{\smash{{\SetFigFont{10}{12.0}{\rmdefault}{\mddefault}{\updefault}{$dq^2$}%
}}}}
\put(3348,-5991){\makebox(0,0)[lb]{\smash{{\SetFigFont{10}{12.0}{\rmdefault}{\mddefault}{\updefault}{$cq^2$}%
}}}}
\put(823,-5615){\makebox(0,0)[lb]{\smash{{\SetFigFont{10}{12.0}{\rmdefault}{\mddefault}{\updefault}{$b$}%
}}}}
\put(1753,-5637){\makebox(0,0)[lb]{\smash{{\SetFigFont{10}{12.0}{\rmdefault}{\mddefault}{\updefault}{$b$}%
}}}}
\put(2713,-5615){\makebox(0,0)[lb]{\smash{{\SetFigFont{10}{12.0}{\rmdefault}{\mddefault}{\updefault}{$b$}%
}}}}
\put(823,-4729){\makebox(0,0)[lb]{\smash{{\SetFigFont{10}{12.0}{\rmdefault}{\mddefault}{\updefault}{$b$}%
}}}}
\put(1768,-4729){\makebox(0,0)[lb]{\smash{{\SetFigFont{10}{12.0}{\rmdefault}{\mddefault}{\updefault}{$b$}%
}}}}
\put(2713,-4729){\makebox(0,0)[lb]{\smash{{\SetFigFont{10}{12.0}{\rmdefault}{\mddefault}{\updefault}{$b$}%
}}}}
\put(749,-5019){\makebox(0,0)[lb]{\smash{{\SetFigFont{10}{12.0}{\rmdefault}{\mddefault}{\updefault}{$dq$}%
}}}}
\put(1399,-5019){\makebox(0,0)[lb]{\smash{{\SetFigFont{10}{12.0}{\rmdefault}{\mddefault}{\updefault}{$cq$}%
}}}}
\put(1635,-5019){\makebox(0,0)[lb]{\smash{{\SetFigFont{10}{12.0}{\rmdefault}{\mddefault}{\updefault}{$dq^2$}%
}}}}
\put(2344,-5019){\makebox(0,0)[lb]{\smash{{\SetFigFont{10}{12.0}{\rmdefault}{\mddefault}{\updefault}{$cq^2$}%
}}}}
\put(2639,-5025){\makebox(0,0)[lb]{\smash{{\SetFigFont{10}{12.0}{\rmdefault}{\mddefault}{\updefault}{$dq^3$}%
}}}}
\put(3289,-5019){\makebox(0,0)[lb]{\smash{{\SetFigFont{10}{12.0}{\rmdefault}{\mddefault}{\updefault}{$cq^3$}%
}}}}
\put(690,-4015){\makebox(0,0)[lb]{\smash{{\SetFigFont{10}{12.0}{\rmdefault}{\mddefault}{\updefault}{$dq^2$}%
}}}}
\put(1340,-4015){\makebox(0,0)[lb]{\smash{{\SetFigFont{10}{12.0}{\rmdefault}{\mddefault}{\updefault}{$cq^2$}%
}}}}
\put(1694,-4015){\makebox(0,0)[lb]{\smash{{\SetFigFont{10}{12.0}{\rmdefault}{\mddefault}{\updefault}{$dq^3$}%
}}}}
\put(2285,-4015){\makebox(0,0)[lb]{\smash{{\SetFigFont{10}{12.0}{\rmdefault}{\mddefault}{\updefault}{$cq^3$}%
}}}}
\put(2639,-4015){\makebox(0,0)[lb]{\smash{{\SetFigFont{10}{12.0}{\rmdefault}{\mddefault}{\updefault}{$dq^4$}%
}}}}
\put(3289,-4080){\makebox(0,0)[lb]{\smash{{\SetFigFont{10}{12.0}{\rmdefault}{\mddefault}{\updefault}{$cq^4$}%
}}}}
\put(1399,-5637){\makebox(0,0)[lb]{\smash{{\SetFigFont{10}{12.0}{\rmdefault}{\mddefault}{\updefault}{$\frac{a}{q}$}%
}}}}
\put(1341,-4731){\makebox(0,0)[lb]{\smash{{\SetFigFont{10}{12.0}{\rmdefault}{\mddefault}{\updefault}{$\frac{a}{q}$}%
}}}}
\put(2286,-4731){\makebox(0,0)[lb]{\smash{{\SetFigFont{10}{12.0}{\rmdefault}{\mddefault}{\updefault}{$\frac{a}{q}$}%
}}}}
\put(3231,-4731){\makebox(0,0)[lb]{\smash{{\SetFigFont{10}{12.0}{\rmdefault}{\mddefault}{\updefault}{$\frac{a}{q}$}%
}}}}
\put(2350,-5615){\makebox(0,0)[lb]{\smash{{\SetFigFont{10}{12.0}{\rmdefault}{\mddefault}{\updefault}{$\frac{a}{q}$}%
}}}}
\put(3295,-5615){\makebox(0,0)[lb]{\smash{{\SetFigFont{10}{12.0}{\rmdefault}{\mddefault}{\updefault}{$\frac{a}{q}$}%
}}}}
\put(5198,-6161){\makebox(0,0)[lb]{\smash{{\SetFigFont{10}{12.0}{\rmdefault}{\mddefault}{\updefault}{$\frac{d}{\Delta}$}%
}}}}
\put(5552,-6161){\makebox(0,0)[lb]{\smash{{\SetFigFont{10}{12.0}{\rmdefault}{\mddefault}{\updefault}{$\frac{c}{\Delta}$}
}}}}
\put(6497,-6161){\makebox(0,0)[lb]{\smash{{\SetFigFont{10}{12.0}{\rmdefault}{\mddefault}{\updefault}{$\frac{c}{\Delta}$}%
}}}}
\put(7442,-6161){\makebox(0,0)[lb]{\smash{{\SetFigFont{10}{12.0}{\rmdefault}{\mddefault}{\updefault}{$\frac{c}{\Delta}$}%
}}}}
\put(5080,-5657){\makebox(0,0)[lb]{\smash{{\SetFigFont{10}{12.0}{\rmdefault}{\mddefault}{\updefault}{$\frac{b}{\Delta}$}%
}}}}
\put(6025,-5657){\makebox(0,0)[lb]{\smash{{\SetFigFont{10}{12.0}{\rmdefault}{\mddefault}{\updefault}{$\frac{b}{\Delta q}$}%
}}}}
\put(5434,-5539){\makebox(0,0)[lb]{\smash{{\SetFigFont{10}{12.0}{\rmdefault}{\mddefault}{\updefault}{$\frac{a}{\Delta q}$}%
}}}}
\put(6134,-6166){\makebox(0,0)[lb]{\smash{{\SetFigFont{10}{12.0}{\rmdefault}{\mddefault}{\updefault}{$\frac{d}{\Delta}$}%
}}}}
\put(7079,-6166){\makebox(0,0)[lb]{\smash{{\SetFigFont{10}{12.0}{\rmdefault}{\mddefault}{\updefault}{$\frac{d}{\Delta}$}%
}}}}
\put(6400,-5420){\makebox(0,0)[lb]{\smash{{\SetFigFont{10}{12.0}{\rmdefault}{\mddefault}{\updefault}{$\frac{a}{\Delta q^2}$}%
}}}}
\put(6850,-5657){\makebox(0,0)[lb]{\smash{{\SetFigFont{10}{12.0}{\rmdefault}{\mddefault}{\updefault}{$\frac{b}{\Delta q^2}$}%
}}}}
\put(7383,-5421){\makebox(0,0)[lb]{\smash{{\SetFigFont{10}{12.0}{\rmdefault}{\mddefault}{\updefault}{$\frac{a}{\Delta q^3}$}%
}}}}
\put(5192,-5161){\makebox(0,0)[lb]{\smash{{\SetFigFont{10}{12.0}{\rmdefault}{\mddefault}{\updefault}{$\frac{d}{\Delta}$}%
}}}}
\put(5546,-5161){\makebox(0,0)[lb]{\smash{{\SetFigFont{10}{12.0}{\rmdefault}{\mddefault}{\updefault}{$\frac{c}{\Delta}$}%
}}}}
\put(6491,-5161){\makebox(0,0)[lb]{\smash{{\SetFigFont{10}{12.0}{\rmdefault}{\mddefault}{\updefault}{$\frac{c}{\Delta}$}%
}}}}
\put(7436,-5161){\makebox(0,0)[lb]{\smash{{\SetFigFont{10}{12.0}{\rmdefault}{\mddefault}{\updefault}{$\frac{c}{\Delta}$}%
}}}}
\put(6128,-5166){\makebox(0,0)[lb]{\smash{{\SetFigFont{10}{12.0}{\rmdefault}{\mddefault}{\updefault}{$\frac{d}{\Delta}$}%
}}}}
\put(7073,-5166){\makebox(0,0)[lb]{\smash{{\SetFigFont{10}{12.0}{\rmdefault}{\mddefault}{\updefault}{$\frac{d}{\Delta}$}%
}}}}
\put(5050,-4749){\makebox(0,0)[lb]{\smash{{\SetFigFont{10}{12.0}{\rmdefault}{\mddefault}{\updefault}{$\frac{b}{\Delta q}$}%
}}}}
\put(5490,-4454){\makebox(0,0)[lb]{\smash{{\SetFigFont{10}{12.0}{\rmdefault}{\mddefault}{\updefault}{$\frac{a}{\Delta q^2}$}%
}}}}
\put(5900,-4690){\makebox(0,0)[lb]{\smash{{\SetFigFont{10}{12.0}{\rmdefault}{\mddefault}{\updefault}{$\frac{b}{\Delta q^2}$}%
}}}}
\put(6457,-4454){\makebox(0,0)[lb]{\smash{{\SetFigFont{10}{12.0}{\rmdefault}{\mddefault}{\updefault}{$\frac{a}{\Delta q^3}$}%
}}}}
\put(5192,-4216){\makebox(0,0)[lb]{\smash{{\SetFigFont{10}{12.0}{\rmdefault}{\mddefault}{\updefault}{$\frac{d}{\Delta}$}%
}}}}
\put(5546,-4216){\makebox(0,0)[lb]{\smash{{\SetFigFont{10}{12.0}{\rmdefault}{\mddefault}{\updefault}{$\frac{c}{\Delta}$}%
}}}}
\put(6491,-4216){\makebox(0,0)[lb]{\smash{{\SetFigFont{10}{12.0}{\rmdefault}{\mddefault}{\updefault}{$\frac{c}{\Delta}$}%
}}}}
\put(7436,-4216){\makebox(0,0)[lb]{\smash{{\SetFigFont{10}{12.0}{\rmdefault}{\mddefault}{\updefault}{$\frac{c}{\Delta}$}%
}}}}
\put(6128,-4221){\makebox(0,0)[lb]{\smash{{\SetFigFont{10}{12.0}{\rmdefault}{\mddefault}{\updefault}{$\frac{d}{\Delta}$}%
}}}}
\put(7073,-4221){\makebox(0,0)[lb]{\smash{{\SetFigFont{10}{12.0}{\rmdefault}{\mddefault}{\updefault}{$\frac{d}{\Delta}$}%
}}}}
\put(6900,-4712){\makebox(0,0)[lb]{\smash{{\SetFigFont{10}{12.0}{\rmdefault}{\mddefault}{\updefault}{$\frac{b}{\Delta q^3}$}%
}}}}
\put(7300,-4476){\makebox(0,0)[lb]{\smash{{\SetFigFont{10}{12.0}{\rmdefault}{\mddefault}{\updefault}{$\frac{a}{\Delta q^4}$}%
}}}}
\end{picture}}

\caption{Illustration of the proof of Lemma \ref{weighttransform}, where $\Delta:=ad+bc$.}
\label{doubletransform2}
\end{figure}
\begin{proof}
The proof is illustrated in the Figure \ref{doubletransform2}, for $m=3$ and $n=4$. First, we apply Vertex-splitting Lemma \ref{VS} to all vertices of
$AR_{m,n}(a,b,c,d,q)$ as in Figures \ref{doubletransform2}(a) and (b). Apply the Spider Lemma  \ref{spider} around $mn$ shaded cells, and remove all edges
 incident to a vertex of degree 1, which is forced edges (see Figure \ref{doubletransform2}(b)). This way, $G\# AR_{m,n}(a,b,c,d,q)$ is transformed into
 $G\# {}_{|}\overline{AR}_{m-\frac{1}{2},n-1}$, where ${}_{|}\overline{AR}_{m-\frac{1}{2},n-1}$ is the weighted version of the graph ${}_{|}AR_{m-\frac{1}{2},n-1}$
 illustrated in Figure \ref{doubletransform2}(c). Next, we divide the graph ${}_{|}\overline{AR}_{m-\frac{1}{2},n-1}$, except for its vertical edges, into $m(n-1)$ subgraphs
 restricted  by the dotted squares as in Figure \ref{doubletransform2}(c). Apply Star Lemma \ref{star} with factor $t=\Delta q^{i+j-2}$ to the central vertex of the  dotted square in row $i$
  (from bottom to top) and column $j$ (from left to right). Finally, we get the graph $G\# {}_{|}AR_{m-\frac{1}{2},n-1}(a/q,b,c,d)$. By Spider, Graph Splitting, and Star Lemmas, we get
\begin{align*}
\M\left(G\#AR_{m,n}(a,b,c,d,q)\right)&=\prod_{1\leq i\leq m}\prod_{1\leq j\leq n}\Delta q^{{i+j-2}}\cdot \M\left(G\# {}_{|}\overline{AR}_{m-\frac{1}{2},n-1}\right)\\
&=\prod_{1\leq i\leq m}\prod_{1\leq j\leq n}\Delta q^{{i+j-2}}  \prod_{1\leq i\leq m}\prod_{1\leq j\leq n-1}(\Delta q^{{i+j-2}})^{-1}\\
&\quad \times \M\left(G\#{}_{|}AR_{m-\frac{1}{2},n-1}(a/q,b,c,d)\right),
\end{align*}
and the lemma follows.
\end{proof}

\section{Weighted Double Aztec rectangles}

Consider the double Aztec rectangle $\mathcal{DR}_{m_1,n_1,k}^{m_2,n_2}$, where its dominoes are weighted by the weight assignment $wt:=wt_{a,b}^{c,d}(q)$ (defined in the previous section).
To specify the weight assignment, we denote by $\mathcal{DR}_{m_1,n_1,k}^{m_2,n_2}(a,b,c,d,q)$ the  weighted double Aztec rectangle. The  tiling generating function of this weighted region is given by the theorem stated below.

\begin{thm}\label{weightdouble}
Assume that $a,b,c,d,q$ are positive numbers. Assume in addition that $m_1,m_2,n_1,n_2,k$ are positive integers so that
 $m_1\leq n_1$, $m_2\leq n_2$, $k\leq \min (m_2,n_2-1)$, and $n_1-m_1=n_2-m_2$. Then
\begin{align*}
\M\Big(\mathcal{DR}_{m_1,n_1,k}^{m_2,n_2}&(a,b,c,d,q)\Big)=c^{(m_2-k+1)(n_1-m_1)}d^{(m_1+k)(n_1-m_1)}\notag\\
&\times \prod_{i=0}^{m_1-1} (\Delta'_i)^{m_1-i}\prod_{i=0}^{m_2-1}(\Delta_i)^{m_2-i} q^{N/2} \prod_{i=1}^{n_1-m_1}\prod_{j=1}^{m_2-k+1}\prod_{t=1}^{m_1+k}\frac{1-q^{i+j+t-1}}{1-q^{i+j+t-2}},
\end{align*}
where $N$ is defined as in Theorem \ref{mainweight},
and where $\Delta'_i=q^{k+m_2+i}(adq^{-i}+bcq)$ and\\ $\Delta_i=q^{m_2+n_2-2-i}(adq^{i}+bc)$.
\end{thm}
We notice that the $a=b=c=d=q=1$ specialization of Theorem \ref{weightdouble} also  implies Corollary \ref{coroweight} in Section 2.
\begin{proof}
Consider the dual graph $G$ of $\mathcal{DR}_{m_1,n_1,k}^{m_2,n_2}(a,b,c,d,q)$.
 Divide the graph $G$ into three parts by two dotted horizontal lines as in Figure \ref{doubletransform3}(a).

 \begin{figure}\centering
\includegraphics[width=14cm]{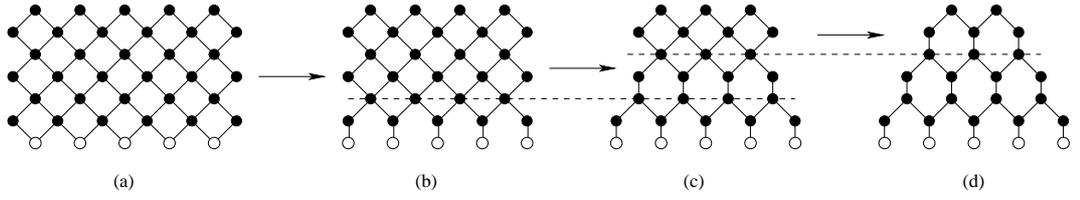}
\caption{Applying the replacement in Lemma \ref{weighttransform} repeatedly to transform the upper part of $G$ into the dual graph of the upper half of weighted hexagon.}
\label{Newholey12}
\end{figure}

\begin{figure}\centering
\includegraphics[width=13cm]{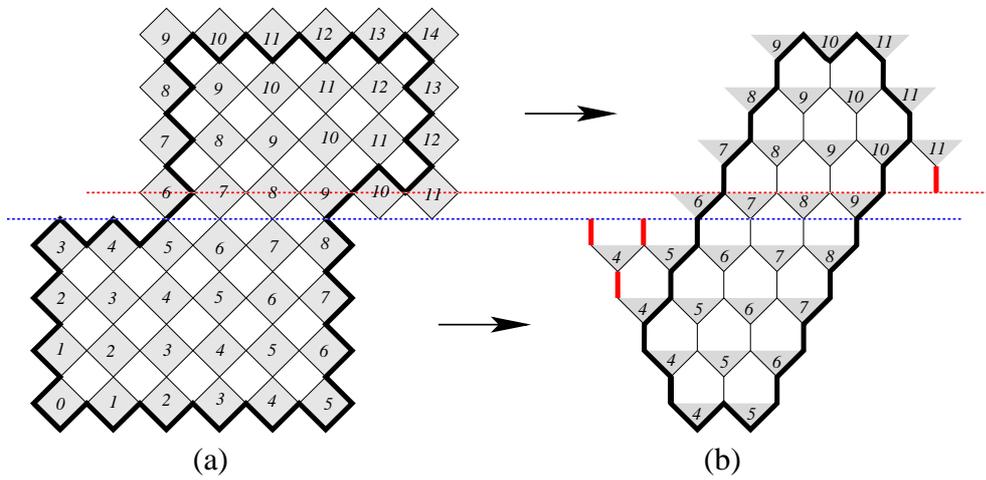}%
\caption{Illustration of the proof of Theorem \ref{mainweight}. The shaded diamond with label $x$ have edge-weight $a,b,dq^x,cq^x$ (in cyclic order, start from the northwest side); and the baseless triangle with label $y$ have the left and right edge-weights $cq^y$ and $dq^y$, respectively.}
\label{doubletransform3}
\end{figure}

Apply the replacement in Lemma \ref{weighttransform} to the top part of $G$, which is isomorphic to  $AR_{m_1,n_1}(a,b,cq^{k+m_2+1},dq^{k+m_2},q)$.
In particular, we replace this portion of $G$ by the graph ${}_{|}AR_{m_1-\frac{1}{2},n_1-1}(a/q,b,cq^{k+m_2+1},dq^{k+m_2},q)$. Viewing $m_1-1$ rows of diamond in the latter graph as the weighted Aztec rectangle graph \\ $AR_{m_1-1,n_1-1}(a/q,b,cq^{k+m_2+2},dq^{k+m_2+1},q)$, we can apply the replacement in Lemma \ref{weighttransform} again.
Keep applying this process as in the Figure \ref{Newholey12} (where the weights are not shown). This way we transform the upper part into a dual graph of the upper half of a weighted hexagon (see the replacement above the dotted lines in Figure \ref{doubletransform3} for the case when $m_1=3,$ $n_1=5$, $m_2=4$,$n_2=6$, $k=3$; the dual graph $G$ is illustrated by the graph restricted by the bold contour in Figure \ref{doubletransform3}(a)).

\begin{figure}\centering
\includegraphics[width=7cm]{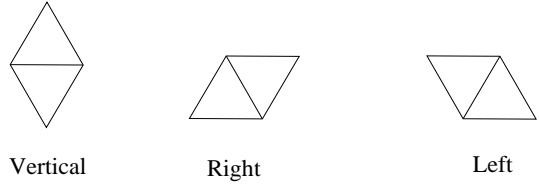}
\caption{Three types of lozenges.}
\label{rhumbustype}
\end{figure}

Similarly, we transform the bottom part of $G$, which is isomorphic to the weighted Aztec rectangle graph
$AR_{m_2,n_2}(a,b,cq^{m_2+n_2-2},dq^{m_2+n_2-2},q^{-1})$ rotated  by $180^{\circ}$, into the dual graph
of the lower half of some weighted hexagon (see the replacement below the dotted lines in Figure \ref{doubletransform3}. Finally, we remove forced vertical edges in the resulting graph.

 This way, the dual graph $G$ of the region is transformed into the dual graph of a weighted lozenge hexagon $\mathcal{H}:=H_{n_1-m_1,m_2-k+1,m_1+k}$
  (see the graph restricted by the bold contour in Figure \ref{doubletransform3}(b)). In particular, the lozenges in $\mathcal{H}$ are weighted as follows. All vertical lozenges are weighted by $1$.
A left lozenge is weighted by $cq^{m_2+i}$, where the Euclidian distance between the left side of the lozenge and the southwest side of the hexagon is $i\sqrt{3}/2$.
Finally,   a right lozenge has weight $cq^{m_2+i}$, where the distance from the lower-left vertex of the lozenge to the southwest side of the hexagon in $i\sqrt{3}/2$.
Figure \ref{volumepartition}(c) shows a weight assignment for the lozenges in a sample tiling of $\mathcal{H}$  for the case $k=2$, $m_1=3$, $m_2=4$, $n_1=7$, $n_2=8$, in which each left lozenge
 with label $x$ is weighted by $cq^x$; and each right lozenge with label $y$ is weighted by $dq^y$, and all vertical lozenges are weighted by $1$.
By the above replacement process and Lemma \ref{weighttransform}, we have
\begin{align}\label{weq1}
\M(G)=&\prod_{i=0}^{m_1-1}(\Delta'_i)^{m_1-i}q^{(m_1-i)(n_1-i-1)+\binom{m_1-i}{2}} \notag\\
&\times \prod_{i=0}^{m_2-1}(\Delta_i)^{m_2-i}q^{-(m_2-i)(n_2-i-1)-\binom{m_2-i}{2}}  \cdot \M(\mathcal{H}),
\end{align}
where $\Delta'_i=q^{k+m_2+i}(adq^{-i}+bcq)$ and $\Delta_i=q^{m_2+n_2-2-i}(adq^{i}+bc)$.

\begin{figure}\centering
\setlength{\unitlength}{3947sp}%
\begingroup\makeatletter\ifx\SetFigFont\undefined%
\gdef\SetFigFont#1#2#3#4#5{%
  \reset@font\fontsize{#1}{#2pt}%
  \fontfamily{#3}\fontseries{#4}\fontshape{#5}%
  \selectfont}%
\fi\endgroup%
\resizebox{13cm}{!}{
\begin{picture}(0,0)%
\includegraphics{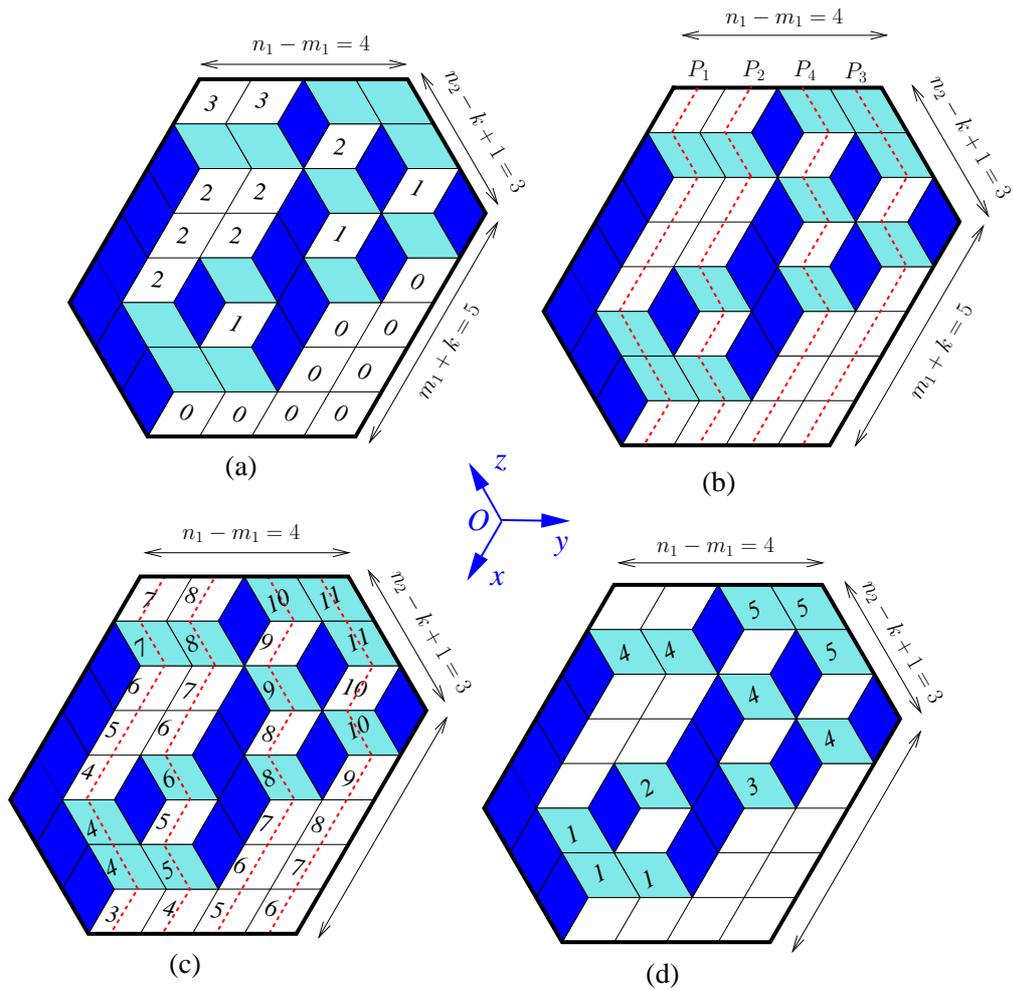}%
\end{picture}%
%
%

\begin{picture}(14750,14940)(328,-14354)
\put(11101,291){\makebox(0,0)[lb]{\smash{{\SetFigFont{20}{24.0}{\rmdefault}{\mddefault}{\itdefault}{$n_1-m_1=4$}%
}}}}
\put(4006,-121){\makebox(0,0)[lb]{\smash{{\SetFigFont{20}{24.0}{\rmdefault}{\mddefault}{\itdefault}{$n_1-m_1=4$}%
}}}}
\put(6886,-481){\rotatebox{300.0}{\makebox(0,0)[lb]{\smash{{\SetFigFont{20}{24.0}{\rmdefault}{\mddefault}{\itdefault}{$n_2-k+1=3$}%
}}}}}
\put(6631,-5401){\rotatebox{60.0}{\makebox(0,0)[lb]{\smash{{\SetFigFont{20}{24.0}{\rmdefault}{\mddefault}{\itdefault}{$m_1+k=5$}%
}}}}}
\put(10531,-511){\makebox(0,0)[lb]{\smash{{\SetFigFont{20}{24.0}{\rmdefault}{\mddefault}{\itdefault}{$P_1$}%
}}}}
\put(11356,-511){\makebox(0,0)[lb]{\smash{{\SetFigFont{20}{24.0}{\rmdefault}{\mddefault}{\itdefault}{$P_2$}%
}}}}
\put(12886,-541){\makebox(0,0)[lb]{\smash{{\SetFigFont{20}{24.0}{\rmdefault}{\mddefault}{\itdefault}{$P_3$}%
}}}}
\put(12121,-511){\makebox(0,0)[lb]{\smash{{\SetFigFont{20}{24.0}{\rmdefault}{\mddefault}{\itdefault}{$P_4$}%
}}}}
\put(2956,-7494){\makebox(0,0)[lb]{\smash{{\SetFigFont{20}{24.0}{\rmdefault}{\mddefault}{\itdefault}{$n_1-m_1=4$}%
}}}}
\put(10066,-7704){\makebox(0,0)[lb]{\smash{{\SetFigFont{20}{24.0}{\rmdefault}{\mddefault}{\itdefault}{$n_1-m_1=4$}%
}}}}
\put(6078,-8044){\rotatebox{300.0}{\makebox(0,0)[lb]{\smash{{\SetFigFont{20}{24.0}{\rmdefault}{\mddefault}{\itdefault}{$n_2-k+1=3$}%
}}}}}
\put(13143,-8194){\rotatebox{300.0}{\makebox(0,0)[lb]{\smash{{\SetFigFont{20}{24.0}{\rmdefault}{\mddefault}{\itdefault}{$n_2-k+1=3$}%
}}}}}
\put(14163,-574){\rotatebox{300.0}{\makebox(0,0)[lb]{\smash{{\SetFigFont{20}{24.0}{\rmdefault}{\mddefault}{\itdefault}{$n_2-k+1=3$}%
}}}}}
\put(14017,-5408){\rotatebox{60.0}{\makebox(0,0)[lb]{\smash{{\SetFigFont{20}{24.0}{\rmdefault}{\mddefault}{\itdefault}{$m_1+k=5$}%
}}}}}
\end{picture}}
\caption{ (a) Each lozenge tiling $T$ of the lozenge hexagon $\mathcal{H}$
 corresponds to stack of unit cubes fitting in a $(n_1-m_1)\times (m_2-k+1) \times (m_1+k)$ box. A white lozenge indicates the
  upper face of a column of  unit cubes, and its label shows the number of unit cubes in the column. (b) Each lozenge in a tiling $T$ of $\mathcal{H}$ corresponds to a family of
  $n_1-m_1$ distinct lozenge-paths $P_i$'s. (c) The weight assignment of lozenges in the tiling $T$  (d) The weights of  lozenges in the tiling $T$
after  the weight changing (the lozenges with no label are weighted by $1$).}
\label{volumepartition}
\end{figure}

Now, each lozenge tiling $T$ of the hexagon $\mathcal{H}$ corresponds to family of $n_1-m_1$ disjoint lozenge-paths consisting of \textit{left}
 and \textit{right lozenge} $\textbf{P}=(P_1,P_2,\dotsc,P_{n-m})$ (see Figure \ref{rhumbustype} for three orientations of a lozenge), where $P_i$ starts from the lozenges containing
 the $i$-th triangle on the north side and ends at the lozenges containing the $i$-th triangle on the south side (see Figure \ref{volumepartition}(b)).

We notice that all lozenges, which are not on the paths $P_i$'s, are vertical and have weight $1$. Thus, the weight of the tiling $T$ equals $wt(P)=\prod_{i=1}^{n_1-m_1}wt(P_i)$, where $wt(P_i)$
is the product of the weights of lozenges on the path $P_i$.

 Each path $P_i$ has $m_2-k$ left lozenges and $m_1+k$ right lozenges. Moreover, if one moves from the bottom of the path $P_i$, the $j$-th right lozenge has weight
  $dq^{m_2+i+j-1}$ (i.e. has label $m_2+i+j-1$  Figure \ref{volumepartition}(c)). Next, we re-assign to each right lozenge a weight $1$, and divide weights of left lozenges on $P_i$ by $cq^{m_2+i-1}$. We get a new weight
   assignment $wt'$ and
\begin{equation*}
wt(P_i)=wt'(P_i)c^{m_2-k}q^{(m_2+i-1)(m_2-k)}d^{m_1+k}\prod_{j=1}^{m_1+k}q^{m_1+m_2+k+i-j},
\end{equation*}
where $wt'(P_i)$ is the new weight of $P_i$. Multiplying the above equations, for $i=1,2,\dotsc,n_1-m_1$, we obtain
\begin{align}\label{weq2}
wt(\textbf{P})=&wt'(\textbf{P})c^{(n_1-m_1)(m_2-k)}q^{(n_1-m_1)m_2(m_2-k)+(m_2-k)\binom{n_1-m_1}{2}}d^{(m_1+k)(n_1-m_1)}\notag\\
&\times\prod_{i=1}^{n_1-m_1}\prod_{j=1}^{m_1+k}q^{m_1+m_2+k+i-j}.
\end{align}

We can see that in the weight assignment $wt'$, all right and vertical lozenges are weighted by $1$; and a left lozenge is weighted by $i$, where $i\sqrt{3}/2$ is the distance between the left side of the lozenge and the southwest side of the hexagon (see Figure \ref{volumepartition}(d) in which each left lozenge with label $x$ is weighted by $q^x$).

Since only left lozenges have weights different from $1$ in the new weight assignment $wt'$,
$wt'(\textbf{P})=q^{s}$, where $s$ is the sum of all labels of left lozenges as in Figure \ref{volumepartition}(d).

Each lozenge tiling $T$ of the lozenge hexagon $\mathcal{H}$
 corresponds to plane partition $\pi$ (a stack of unit cubes) fitting in a $(n_1-m_1)\times (m_2-k+1) \times (m_1+k)$ box (see Figure \ref{volumepartition}(a)). However, one readily sees that the sum of labels $s$ gives us the number of unit cubes in the complement $\pi^c$ of the plane partition $\pi$.
Taking sum over all tilings $T$, by (\ref{weq2}), we get
\begin{align*}\label{weq4}
\sum_{T}wt(T)&=\sum_{\textbf{P}}wt(\textbf{P})=K\sum_{\textbf{P}}wt'(\textbf{P})\\
&=K\prod_{\pi}q^{|\pi^c|}=K\prod_{\pi}q^{|\pi|},
\end{align*}
where the products are taken over all plane partitions $\pi$ fitting in a $(n_1-m_1)\times(m_2-k+1)\times (m_1+k)$ box, and where
\[K=c^{(m_2-k+1)(n_1-m_1)}d^{(m_1+k)(n_1-m_1)}q^{\frac{(n_1-m_1)}{2}(2m_2^2+m_2m_1+m_2n_1+k^2+2km_1+m_1n_1+ k-m_2)}.\]
By MacMahon's theorem (\ref{qMac}), we have
\begin{equation*}\label{weq5}
\sum_{T}wt(T)=K\prod_{i=1}^{n_1-m_1}\prod_{j=1}^{m_2-k+1}\prod_{t=1}^{m_1+k}\frac{[i+j+t-1]_q}{[i+j+t-2]_q}.
\end{equation*}
Then the theorem follows from (\ref{weq1}).
\end{proof}

\section{Proof of Theorem \ref{mainweight}}

Similar to the case of lozenge tilings of a hexagon, the domino tilings of the double Aztec rectangle $\mathcal{DR}^{m_2mn_2}_{m_1,n_1,k}$ correspond to families of non-intersecting paths as follows.

Label the centers of vertical steps on the lower south-western boundary of the region by $u_1,u_2,\dots,u_{m_2}$, and the centers of vertical steps on the upper north-western boundary by
$u_{m_2+1},$ $u_{m_2+2},$ $\dotsc,u_{m_2+n_1}$. Next, we label the centers of vertical steps on the lower south-eastern boundary by $v_1,v_2,\dotsc,v_{n_2}$, and the centers of
 vertical steps on the upper north-eastern boundary by $v_{n_2+1},v_{n_2+2},\dotsc,v_{n_2+m_1}$ (see Figure \ref{double6} for the case when $m_1=4$, $n_1=8$, $m_2=3$, $n_2=7$, $k=2$).

\begin{figure}\centering
\includegraphics[width=10cm]{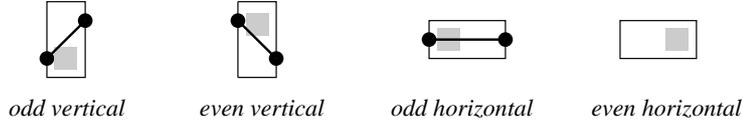}
\caption{Decorating the dominoes by steps of Schr\"{o}der paths.}
\label{Drawdomino}
\end{figure}

We color the region black and white so that two squares sharing an edge have opposite colors and that the unit squares along the southwest side $\mathcal{AR}_{m_1,n_1}$ are white.
Given a tiling $T$ of the double Aztec rectangle,  we decorate its dominoes as in Figure \ref{Drawdomino}. Then the tiling $T$ corresponds to a family of non-intersecting
 paths $\textbf{P}=(P_1,P_2,\dotsc,P_{m_1+n_2})$, where $P_i$ connects $u_i$ and $v_i$ (see Figure \ref{double7}).

We notice that the path $P_i$ is a part of a \emph{Schr\"{o}der path} (a lattice path on the square lattice starting and finishing on the $x$-axis, and using only $(1,1)$, $(1,-1)$ and
 $(0,2)$ steps so that it  never go below $x$-axis). We also remark that the above correspondence between domino tilings and families of non-intersecting Sch\"{o}der paths is inspired by Eu and Fu's correspondence for the case of the Aztec diamonds \cite{Eu}.
\begin{figure}\centering
\resizebox{!}{7cm}{
\begin{picture}(0,0)%
\includegraphics{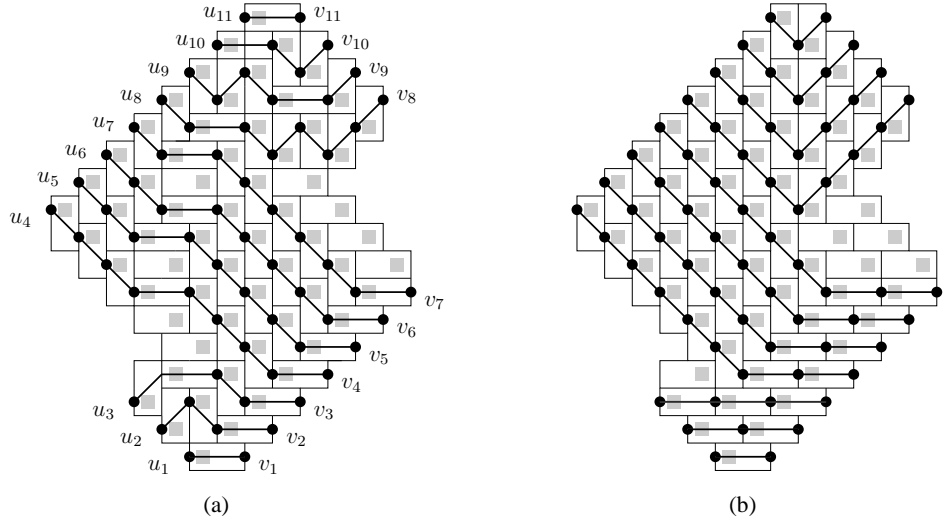}%
\end{picture}%
\setlength{\unitlength}{3947sp}%
\begingroup\makeatletter\ifx\SetFigFont\undefined%
\gdef\SetFigFont#1#2#3#4#5{%
  \reset@font\fontsize{#1}{#2pt}%
  \fontfamily{#3}\fontseries{#4}\fontshape{#5}%
  \selectfont}%
\fi\endgroup%
\begin{picture}(7983,4543)(1285,-3842)
\put(2481,-3413){\makebox(0,0)[lb]{\smash{{\SetFigFont{12}{14.4}{\rmdefault}{\mddefault}{\updefault}{$u_1$}%
}}}}
\put(2245,-3177){\makebox(0,0)[lb]{\smash{{\SetFigFont{12}{14.4}{\rmdefault}{\mddefault}{\updefault}{$u_2$}%
}}}}
\put(2009,-2941){\makebox(0,0)[lb]{\smash{{\SetFigFont{12}{14.4}{\rmdefault}{\mddefault}{\updefault}{$u_3$}%
}}}}
\put(1300,-1287){\makebox(0,0)[lb]{\smash{{\SetFigFont{12}{14.4}{\rmdefault}{\mddefault}{\updefault}{$u_4$}%
}}}}
\put(1537,-933){\makebox(0,0)[lb]{\smash{{\SetFigFont{12}{14.4}{\rmdefault}{\mddefault}{\updefault}{$u_5$}%
}}}}
\put(1773,-696){\makebox(0,0)[lb]{\smash{{\SetFigFont{12}{14.4}{\rmdefault}{\mddefault}{\updefault}{$u_6$}%
}}}}
\put(2009,-460){\makebox(0,0)[lb]{\smash{{\SetFigFont{12}{14.4}{\rmdefault}{\mddefault}{\updefault}{$u_7$}%
}}}}
\put(2245,-224){\makebox(0,0)[lb]{\smash{{\SetFigFont{12}{14.4}{\rmdefault}{\mddefault}{\updefault}{$u_8$}%
}}}}
\put(2482, 12){\makebox(0,0)[lb]{\smash{{\SetFigFont{12}{14.4}{\rmdefault}{\mddefault}{\updefault}{$u_{9}$}%
}}}}
\put(2718,249){\makebox(0,0)[lb]{\smash{{\SetFigFont{12}{14.4}{\rmdefault}{\mddefault}{\updefault}{$u_{10}$}%
}}}}
\put(2954,485){\makebox(0,0)[lb]{\smash{{\SetFigFont{12}{14.4}{\rmdefault}{\mddefault}{\updefault}{$u_{11}$}%
}}}}
\put(3899,485){\makebox(0,0)[lb]{\smash{{\SetFigFont{12}{14.4}{\rmdefault}{\mddefault}{\updefault}{$v_{11}$}%
}}}}
\put(4135,248){\makebox(0,0)[lb]{\smash{{\SetFigFont{12}{14.4}{\rmdefault}{\mddefault}{\updefault}{$v_{10}$}%
}}}}
\put(4371, 12){\makebox(0,0)[lb]{\smash{{\SetFigFont{12}{14.4}{\rmdefault}{\mddefault}{\updefault}{$v_9$}%
}}}}
\put(4607,-224){\makebox(0,0)[lb]{\smash{{\SetFigFont{12}{14.4}{\rmdefault}{\mddefault}{\updefault}{$v_8$}%
}}}}
\put(4844,-1996){\makebox(0,0)[lb]{\smash{{\SetFigFont{12}{14.4}{\rmdefault}{\mddefault}{\updefault}{$v_7$}%
}}}}
\put(4607,-2232){\makebox(0,0)[lb]{\smash{{\SetFigFont{12}{14.4}{\rmdefault}{\mddefault}{\updefault}{$v_6$}%
}}}}
\put(4371,-2468){\makebox(0,0)[lb]{\smash{{\SetFigFont{12}{14.4}{\rmdefault}{\mddefault}{\updefault}{$v_5$}%
}}}}
\put(4135,-2704){\makebox(0,0)[lb]{\smash{{\SetFigFont{12}{14.4}{\rmdefault}{\mddefault}{\updefault}{$v_4$}%
}}}}
\put(3899,-2941){\makebox(0,0)[lb]{\smash{{\SetFigFont{12}{14.4}{\rmdefault}{\mddefault}{\updefault}{$v_3$}%
}}}}
\put(3663,-3177){\makebox(0,0)[lb]{\smash{{\SetFigFont{12}{14.4}{\rmdefault}{\mddefault}{\updefault}{$v_2$}%
}}}}
\put(3426,-3413){\makebox(0,0)[lb]{\smash{{\SetFigFont{12}{14.4}{\rmdefault}{\mddefault}{\updefault}{$v_1$}%
}}}}
\end{picture}}
\caption{(a) A  domino tiling of a double Aztec rectangle corresponds to a family of non-intersecting (partial) Schr\"{o}der paths. (b) The correspondence for the minimal tiling $T_0$.}
\label{double7}
\end{figure}

\begin{figure}\centering
\includegraphics{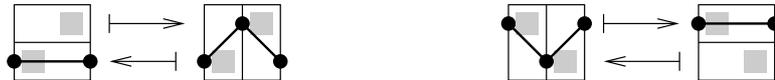}
\caption{An elementary move raises the rank of the
tiling $T$ by one (left-to-right, respectively) if and only if the corresponding family $\textbf{P}$ of non-intersecting Schr\"{o}der paths increases the underneath area by one.}
\label{elementmove}
\end{figure}

As shown in Figure \ref{elementmove}, each elementary move rising the rank of the tiling $T$ by one gives a deformation of some path in $\textbf{P}$
 increasing the underneath area by one (here, the ``ground" is the horizontal line passing $u_1$ and $v_1$). Thus, the difference between the rank of a tiling
 $T$ and the total underneath area of its corresponding path family $\textbf{P}$ does \emph{not} depend on the choice of $T$. It is easy to see the path family $\textbf{P}_0$
  corresponding to $T_0$ has the smallest total underneath area (see Figure \ref{double7}(b)). That also explains the choice of $T_0$ as the \emph{minimal} tiling of the region in Section 2.

\begin{proof}[Proof of Theorem \ref{mainweight}]
We apply the weight assignment $wt_{1,1}^{t,q}(q^2)$ to the dominoes of the double Aztec rectangle $\mathcal{DR}^{m_2,n_2}_{m_1,n_1,k}$. Consider a family of non-interesting (partial)
Schr\"{o}der paths $\textbf{P}$ corresponding to a tiling $T$ of the region. Denote by $\alpha(T)$ we exponent of $q$ in the expression of the weight $wt(\textbf{P})$.
We have $r(T)-\alpha(T)$ is a constant that does not depend on $T$. Thus,
\begin{equation}\label{finaleq1}
r(T)-\alpha(T)=r(T_0)-\alpha(T_0)=-\alpha(T_0).
\end{equation}

By an explicit evaluation, one gets
\small{\begin{align}\label{finaleq2}
\alpha(T_0)&=\frac{2m_2(m_2-1)(m_2+1)}{3}+(m_2-k+1)(m_2+n_2-1)(n_2-m_2)\notag\\
&\quad \quad +\sum_{i=0}^{m_1-1}(2(m_1-i)(k+n_2+2i) +(m_1-i)^2)\notag\\
&=\frac{2m_2(m_2-1)(m_2+1)}{3}+(m_2-k+1)(m_2+n_2-1)(n_1-m_1)\notag\\
&\quad \quad +\frac{m_1(m_1+1)(2k+2m_1+2n_2-1)}{2}.
\end{align}}

\normalsize Considering the numbers of up, down and level steps  in  the family of lattice paths $\textbf{P}$ (denoted by $\up(\textbf{P})$, $\down(\textbf{P})$, and $\level(\textbf{P})$, respectively), we have
\begin{align*}
\up(\textbf{P})+\down(\textbf{P})+2\level(\textbf{P})&=m_2(m_2+1)+2(n_1-m_1)(m_2-k+1)\notag\\
&+(n_1-m_1)(m_1+k)+m_1(m_1+1),
\end{align*}
so
\begin{align}\label{finaleq3}
v(T)&=\up(\textbf{P})+\down(\textbf{P})\notag\\
&=\frac{m_2(m_2+1)}{2}+(n_1-m_1)(m_2-k+1)+\frac{(n_1-m_1)(m_1+k)}{2}\notag\\
&\quad+\frac{m_1(m_1+1)}{2}-\level(\textbf{P}).
\end{align}
By (\ref{finaleq1}), (\ref{finaleq2}) and (\ref{finaleq3}), we have
\begin{align*}
\sum_{T}t^{v(T)}q^{r(T)}&=t^{\frac{m_2(m_2+1)}{2}+(n_1-m_1)(m_2-k+1)+\frac{(n_1-m_1)(m_1+k)}{2}+\frac{m_1(m_1+1)}{2}}q^{\alpha(T_0)}\sum_{\textbf{P}}t^{-\level(\textbf{P})}q^{\alpha(T)}\notag\\
&=t^{\frac{m_2(m_2+1)}{2}+(n_1-m_1)(m_2-k+1)+\frac{(n_1-m_1)(m_1+k)}{2}+\frac{m_1(m_1+1)}{2}}q^{\alpha(T_0)}\notag\\
& \quad \quad\times\M\left(\mathcal{DR}_{m_1,n_1,k}^{m_2,n_2}(1,1,t^{-1},q,q^2)\right),
\end{align*}
and the theorem follows from Theorem \ref{weightdouble}.
\end{proof}



\begin{thebibliography}{10}


\bibitem{Ciucu1}
M. Ciucu,
\emph{Enumeration of perfect matchings in graphs with reflective symmetry},
J. Combin. Theory Ser. A \textbf{77} (1997), 67--97.

\bibitem{Elkies1}
N. Elkies, G. Kuperberg, M.Larsen, and J. Propp, \emph{Alternating-sign matrices and domino tilings (Part I)},  J. Algebraic Combin. \textbf{1}  (1992),  111--132.
\bibitem{Elkies2}
N. Elkies, G. Kuperberg, M.Larsen, and J. Propp, \emph{Alternating-sign matrices and domino tilings (Part II)},  J. Algebraic Combin. \textbf{1}  (1992),  219--234.

\bibitem{Eu}
S.-P. Eu and T.-S. Fu,
\textit{A simple proof of the Aztec diamond theorem}, Electron. J. Combin. \textbf{12} (2005),  R18.

\bibitem{Kamioka}
S. Kamioka,
\emph{Laurent biorthogonal polynomials, $q$-Narayana polynomials and domino tilings of the Aztec diamonds},
J. Combin. Theory Ser. A \textbf{123} (2004), 14--29.


\bibitem{Tri1}
T. Lai.
\emph{Enumeration of hybrid domino-lozenge tilings}, J. Combin. Theory Ser. A \textbf{122} (2014),  53--81.

\bibitem{Tri2}
T. Lai,
\emph{Enumeration of tilings of quartered Aztec rectangles}, Electron. J. Combin.  \textbf{21}(4) 2014, P4.46. 

\bibitem{Tri3}
T. Lai,
\textit{Generating function of the tilings of an Aztec rectangle with holes}, to appear in Graphs and Combinatorics. 
Preprint: arXiv:1402.0825v6.


\bibitem{Mac}
P. A. MacMahon.
\newblock Combinatory Analysis.
\newblock \textit{Cambridge Univ. Press, 1916, reprinted by Chelsea, New York, 1960.}

\bibitem{propp}
J. Propp,
\textit{Generalized domino-shuffling},
Theoret. Comput. Sci., \textbf{303} (2003), 267--301.

\end{thebibliography}
\end{document}